\documentclass{conm-p-l}%
\usepackage{amsfonts}
\usepackage{graphicx}
\usepackage{subfigure}
\usepackage{amssymb,amsmath,amsthm,amscd}
\usepackage{pstricks}
\usepackage{pst-plot}
\usepackage{latexsym}
\usepackage{tikz}
\usepackage{amsmath}
\usepackage{amssymb}%
\providecommand{\U}[1]{\protect\rule{.1in}{.1in}}
\copyrightinfo{2008}{}
\providecommand{\U}[1]{\protect\rule{.1in}{.1in}}
\usetikzlibrary{arrows,decorations.markings}
\newtheorem{theorem}{Theorem}[section]
\newtheorem{lemma}[theorem]{Lemma}
\theoremstyle{definition}

\theoremstyle{remark}
\newtheorem{remark}[theorem]{Remark}
\numberwithin{equation}{section}

\newcommand{\ra}{\rightarrow}

\renewcommand{\O}{{\mathcal O}}

\begin{document}
\title[Sommerfeld Paradox]{A Resolution of the Sommerfeld Paradox}
\author{Y. Charles Li}
\address{Department of Mathematics, University of Missouri, Columbia, MO 65211, USA}
\email{liyan@missouri.edu}
\author{Zhiwu Lin}
\address{School of Mathematics, Georgia Institute of Technology, Atlanta, GA 30332, USA}
\email{zlin@math.gatech.edu}
\thanks{}
\subjclass{Primary 76, 35, 37; Secondary 34}
\date{}
\dedicatory{ }

\begin{abstract}
Sommerfeld paradox roughly says that mathematically Couette linear shear is
linearly stable for all Reynolds number, but experimentally arbitrarily small
perturbations can induce the transition from the linear shear to turbulence
when the Reynolds number is large enough. The main idea of our resolution of
this paradox is to show that there is a sequence of linearly unstable shears
which approaches the linear shear in the kinetic energy norm but not in the
enstrophy (vorticity) norm. These oscillatory shears are single Fourier modes
in the Fourier series of all the shears. In experiments, such linear
instabilities will manifest themselves as transient nonlinear growth leading
to the transition from the linear shear to turbulence no matter how small the
intitial perturbations to the linear shear are. Under the Euler dynamics,
these oscillatory shears are steady, and cat's eye structures bifurcate from
them as travelling waves. The 3D shears $U(y,z)$ in a neighborhood of these
oscillatory shears are linearly unstable too. Under the Navier-Stokes
dynamics, these oscillatory shears are not steady rather drifting slowly. When
these oscillatory shears are viewed as frozen, the corresponding
Orr-Sommerfeld operator has unstable eigenvalues which approach the
corresponding inviscid eigenvalues when the Reynolds number tends to infinity.
All the linear instabilities mentioned above offer a resolution to the
Sommerfeld paradox, and an initiator for the transition from the linear shear
to turbulence.

\end{abstract}
\maketitle
\tableofcontents

\section{Introduction}

The most influential paradox in fluids is the d'Alembert paradox saying that a
body moving through water has no drag as calculated by d'Alembert \cite{dAl52}
via inviscid theory, while experiments show that there is a substantial drag
on the body. The paradox splitted the field of fluids into two branches: 1.
Hydraulics --- observing phenomena without mathematical explanation, 2.
Theoretical Fluid Mechanics --- mathematically predicting phenomena that could
not be observed. A revolutionary development of the boundary layer theory by
Ludwig Prandtl in 1904 resolved the paradox by paying attention to the
substantial effect of small viscosity in the boundary layer. Prandtl's
boundary layer theory laid the foundation of modern unified fluid mechanics.

Sommerfeld paradox has the potential of being the next most influential
paradox in fluids. The paradox says that the linear shear in Couette flow is
linearly stable for all Reynolds numbers as first calculated by Sommerfeld
\cite{Som08}, but experiments show that any small perturbation size to the
linear shear can lead to the transition from the linear shear to turbulence
when the Reynolds number is large enough. This paradox is the key for
understanding turbulence inside the infinite dimensional phase space.
Dynamical system studies on the Navier-Stokes flow in an infinite dimensional
phase space is still at its developing stage. In this article, we shall
conduct such a study to offer a resolution to the Sommerfeld paradox. Linear
hydrodynamic stability is a classical subject, for a modern version with
dynamical system flavor, see \cite{SH01}.

Couette flow between two parallel horizontal plates is the simplest of all
classical fluid flows with boundary layers. It is one of the most fundamental
flows for understanding the transition to turbulence. Another basic flow is
pipe Poiseuille flow, which is also linearly stable for any Reynolds number as
shown by numerical computations. But Reynolds's famous experiment in 1883
showed the transition to turbulence for large Reynolds number. Indeed, there
are lots of similar features in the study of transient turbulence of plane
Couette flow and pipe Poiseuille flow \cite{Ker05}. We expect that some of our
studies for Couette flow in this paper could be useful for understanding the
turbulence of pipe Poiseuille flow.

The linear stability of Couette flow was first studied by Sommerfeld in 1908
\cite{Som08} using a single Fourier mode analysis to the linearized
Navier-Stokes equations. Sommerfeld found that all eigenvalues are
non-positive for all values of the Reynolds number, indicating the absence of
exponentially growing eigen modes, and concluded that the linear shear is
linearly stable for all Reynolds numbers. This fact was rigorously proved by
Romanov \cite{Rom73} who showed that all the eigenvalues are less than $-C/R$
where $R$ is the Reynolds number and $C$ is a positive constant; furthermore,
the linear shear is nonlinearly asymptotically stable in $L^{2}$ norm of
vorticity. On the other hand, experimentally no matter how small the initial
perturbation to the linear shear is, a transition to turbulence always occurs
when the Reynolds number is large enough.

Now we briefly comment on some previous attempts to explain Sommerfeld
paradox. One popular resolution, which was first suggested by Orr \cite{orr}
(see also \cite{treffeton}), is to use the non-normality of the linearized
Navier-Stokes operator to get algebraic growth of perturbations before their
final decay. (Note: non-normality refers to operators with non-orthogonal
eigenfunctions.) However, it is not clear how such linear algebraic growth
relates to the nonlinear dynamics (see \cite{waleffe-95}). Moreover, the
non-normality theory cannot explain many coherent structures observed in the
transient turbulence.

We believe that norms play a fundamental role in resolving the paradox.
Perturbations with large $L^{2}$ norm of vorticity may still have small
$L^{2}$ norm of velocity, i.e., small energy. In experimental or numerical
studies, such perturbations of small energy are still considered to be small,
but they are outside the vorticity's $L^{2}$ neighborhood of the linear shear
where Romanov's nonlinear stability result is valid \cite{Rom73}. Therefore,
such perturbations have the potential of being linearly unstable and
initiating the transition to turbulence.

The main idea of our resolution is to show the existence of a sequence of
linearly unstable shears which approach the linear shear in the velocity
variable but not in the vorticity variable. These shears are the single modes
of the Fourier series of all the 2D shears of the Couette flow $y + \sum
_{m=1}^{+\infty} c_{m} \sin(my)$. More precisely, our sequence of oscillatory
shears has the form
\begin{equation}
U_{n}\left(  y\right)  =y+\frac{A}{n}\sin(4n\pi y),\ \left(  \frac{1}{2}%
\frac{1}{4\pi}<A<\frac{1}{4\pi}\right)  . \label{formula-U_n}%
\end{equation}
As $n\rightarrow\infty$, the oscillatory shears approach the linear shear,
i.e. $U_{n}(y)\rightarrow y$ in $L^{2}$ and $L^{\infty}$. On the other hand,
in the vorticity variable, the oscillatory shears do not approach the linear
shear since $\partial_{y}U_{n}(y)=1+4A\pi\cos(4n\pi y)\not \ra 1$ in any
Lebesgue norm. Thus in the velocity variable, the oscillatory shears can be
viewed as the linear shear plus small noises. For any large $n$, we prove that
$U_{n}(y)$ is linearly unstable for both inviscid and slightly viscous fluids.
More precisely, in Theorems \ref{OST} and \ref{LOS}, it is shown that these
shears are linearly (exponentially) unstable for both Euler equations and
Navier-Stokes equations with large Reynolds numbers. Under the Navier-Stokes
dynamics, our shears are not steady rather drifting slowly. By viewing them as
frozen, we investigate the spectra of the corresponding linearized
Navier-Stokes operator (Orr-Sommerfeld operator). Moreover, numerical
simulations \cite{Lan09} indicate that its unstable growth rate does not
depend on $n$ substantially, implying that as $n\rightarrow\infty$, its
unstable growth rate does not shrink to zero rather approach a positive
number. Such a linear instability will generate the transition to turbulence
no matter how small the initial perturbation added to the linear shear is, as
long as the Reynolds number is large enough to realize the linear instability.
Numerical simulations \cite{Lan09} indicates that initial perturbations to
$U_{n}(y)$ indeed lead to a transient nonlinear growth. In fact, our new theory 
proposed in \cite{Lan09} claims that such a transient nonlinear growth induced
by slowly drifting states is the only mechanism for transition from the linear 
shear to turbulence.

Based upon comments from our colleagues, it is important to clarify a few
points here:

\begin{enumerate}
\item A more precise formulation of the Sommerfeld paradox is: One one hand,
the linear shear is linearly stable. On the other hand, experiments indicate
that for any $\varepsilon> 0$, there is a $\delta>0$ such that when the
Reynolds number is larger than $\frac{1}{\delta}$, there is a perturbation of
size $\varepsilon$ that leads to the transition from the linear shear to
turbulence. Due to various difficulties and lack of a precise mathematical
direction, experiments in the past were often less conclusive and confusing.
But recent experiments gradually converge to the conclusion of the precise
mathematical statement above.

\item It is fundamental to notice that our sequence of oscillatory shears
(\ref{formula-U_n}) is uniformly unstable for all positive integers $n$, i.e.
for \textit{any} $n=1,2,\cdots$, $U_{n}$ is linearly unstable.

\item Under the Navier-Stokes dynamics, our sequence of oscillatory shears
drifts as follows:
\begin{equation}
U_{n} (t,y) = y+ e^{-\epsilon(4n\pi)^{2}t}\frac{A}{n}\sin(4n\pi y).
\label{dros}%
\end{equation}
In terms of the above Sommerfeld paradox, for any $\varepsilon$, we can find a
$n$ such that $\frac{A}{n} < \varepsilon$, then there is a $\delta$ such that
when the Reynolds number is larger than $\frac{1}{\delta}$ (i.e. $\epsilon<
\delta$), the Orr-Sommerfeld (linear Navier-Stokes) operator has an unstable
eigenvalue at such a $U_{n}$ (\ref{formula-U_n}). Now we inspect the drifting
(\ref{dros}) for such $U_{n}$:
\[
n > \frac{A}{\varepsilon}, \ \epsilon< \delta.
\]
The dual effect of $n$ and $\epsilon$ above can prevent the drifting exponent
$\epsilon(4n\pi)^{2}$ to be too large, therefore prevent the oscillatory
component of $U_{n}$ to be quickly dissipated. In fact, When the Reynolds
number is larger enough (i.e. $\delta$ is small enough), the drifting exponent
$\epsilon(4n\pi)^{2}$ is very small, therefore the drifting and the
dissipation the oscillatory component are very slow. The linear instability of
$U_{n}$ then has plenty of time to be amplified, leading to the observations
in the experiments, and the resolution of the Sommerfeld paradox. Finally,
even when $n =1$, the amplitude of the oscillatory component of $U_{n}$ is
already very small $0.04 < A < 0.08$.
\end{enumerate}

In Theorem \ref{thm-catseye}, we prove the bifurcation to nontrivial
travelling solutions to 2D Euler equation, near the oscillatory shear
$U_{n}(y)$ in the energy norm. The streamlines of these travelling waves have
the structure of Kelvin's cat's eyes.
%Zlin-0510
This study has been recently extended in \cite{lin-zeng-couette} to show that
(vorticity) $H^{\frac{3}{2}}$ is the critical regularity for nontrivial Euler
traveling waves to exist near Couette flow (see Remark \ref{rmk-TW-critical}).
In Theorem \ref{tm-robust}, we also show that 3D shears $\left(  U\left(
y,z\right)  ,0,0\right)  $ in a neighborhood ($W^{1,p}$ ($p>2$) in the
velocity variable) of any linearly unstable 2D shear (including our $U_{n}%
(y)$) are linearly unstable too.
%Zlin-0510
This shows that the instability found near Couette flow in Theorem \ref{OST}
is also robust in the $3$D setting.

In recent years, there has been a renaissance in numerical dynamical system
studies on fluids. The focus was upon three classical flows: plane Couette
flow, plane Poiseuille flow, and pipe Poiseuille flow. Basic flows of them are
the linear shear for plane Couette flow, and the parabolic shear for both
plane and pipe Poiseuille flows. Original studies on transition to turbulence
for these flows were conducted by the Orszag group \cite{OK80} \cite{OP80}.
Primitive steak-roll-wave coherent structures were discovered. The Orszag
group also emphasized the importance of the slowly drifting states like our
$U_{n}(y)$. In 1990, Nagata made a breakthrough by discovering 3D steady
states (fixed points) in plane Couette flow \cite{Nag90}. There are two
branches of 3D steady states, called upper and lower branches. The same 3D
steady states also appeared in plane Poiseuille flow \cite{Wal03}. Most
recently the same 3D steady states were also discovered in pipe Poiseuille
flow and they become traveling waves with constant speeds in the laboratory
\cite{Eck08} \cite{Ker05}. The pipe discovery made laboratory experiment easy
to implement. Indeed, Hof et al. observed such traveling waves in experiments
\cite{Hof04}. A steady state (fixed point) study is generally the starting
point of a full dynamical system study which is still in progress \cite{KK01}
\cite{GHC08} \cite{HGCV09} \cite{Vis07} \cite{SYE06} \cite{SGLLE08}. The most
intriguing new discovery on the steady states is that the lower branch steady
states of the three classical flows share some universal feature, i.e. they
all approach 3D shears as the Reynolds number approaches infinity \cite{WGW07}
\cite{Vis08} \cite{Wal03}. Such 3D shears ($U(y,z), 0, 0$) are neutrally
stable under the linearized Euler dynamics. So they are not our 3D shears in
the neighborhood of our $U_{n}(y)$. On the other hand, any 3D shear
($U(y,z),0,0$) is a steady state of the 3D Euler equations; so it is
interesting to understand the special features of the limiting 3D shear of the
lower branch. A necessary condition on such a limiting 3D shear was also
discovered \cite{LV09}. In numerical and experimental studies on transition to
turbulence, one often observes the steak-roll-wave structure. The streak and
roll are the $0$-th Fourier mode, while the wave is the higher Fourier mode in
the Fourier series of the velocity field. When plotting these modes
separately, the steak-roll-wave structure should generically be observed.
Taking into account the linear instability of our 3D shears in a neighborhood
of our $U_{n}(y)$, such steak-roll-wave structures will be generated too.

The 3D steady state argument on the Sommerfeld paradox claims that the stable
manifolds of the 3D steady states may get very close (they cannot be
connected) to the linear shear in the infinite dimensional phase space, which
leads to the transition from the linear shear to turbulence \cite{GHC08}. The
main criticism of this argument is that these 3D steady states are not close
to the linear shear in the phase space even in the kinetic energy norm; while
experiments show that no matter how small the initial perturbation to the
linear shear is, a transition always occurs when the Reynolds number is large
enough. We believe that the linear instability of our $U_{n}(y)$ (and 3D
shears in its neighborhood) offers a better explanation for the initiation of
transition from the linear shear to turbulence. In other words, the key point
for answering the Sommerfeld paradox is whether or not there is a linear
instability happening arbitrarily close to the linear shear. Our $U_{n}(y)$
(and 3D shears in its neighborhood) does the job. Again, they approach the
linear shear in the velocity variable but not the vorticity variable, in
consistency with Romanov's nonlinear stability theorem \cite{Rom73}. On the
other hand, stable manifolds of the 3D steady states can often only be
established in higher Sobolev spaces in which Romanov's nonlinear stability
theorem prohibits them to get close to the linear shear.

Explorations on two dimensional viscous steady states turn out to be not
successful so far \cite{CE97} \cite{ENR08}. That is, the counterpart of the 3D
upper or lower branch steady state has not been found in 2D. A formal analysis
in \cite{Li10} confirms the non-existence of a viscous steady state in 2D. On
the other hand, numerics shows that transitions still occur from the linear
shear to turbulence in 2D. This further confirms that the stable manifolds of
the 3D steady states are not the initiators for the transition. Under the 2D
Euler dynamics, Theorem \ref{thm-catseye} shows the existence of 2D inviscid
steady states with cat's eye structures. The neighborhood of these eye
structures might be a good place for a future numerical search of 2D viscous
steady states.

A rigorous proof on the existence of 3D steady states in the three classical
flows remains open. The main difficulty is the lack of a proper bifurcation
point. There is no hope directly from the linear shear end since it is
linearly stable. From the infinite Reynolds number limiting shear of the lower
branch, it is promising and challenging since we do not know much of its
property except a necessary condition \cite{LV09}. However, with the discovery
of more and more 3D N-S steady states (travelling waves) numerically besides
the upper and lower branches \cite{HGC09}, it is unclear if any of these 3D
steady states are fundamentally important to the development of transient
turbulence. In Theorem \ref{thm-catseye}, we show that nontrivial 2D steady
states bifurcate from the unstable shears $\left(  U_{n}\left(  y\right)
,0\right)  $ under the 2D Euler dynamics. Since $U_{n}\left(  y\right)  $ is
arbitrarily close to the linear shear in the velocity variable, so we get
inviscid steady states near the linear shear. These inviscid steady states
could provide a natural starting point for constructing steady states of
Navier-Stokes equation with large Reynolds number. This is a problem of
bifurcation from infinity or asymptotic bifurcation theory as suggested by
Yudovich \cite{yudovich}. Of course, to get N-S travelling waves from Euler
travelling waves, the key and difficult issue is to understand the boundary layers.

Proving the existence of chaos (turbulence) for Couette flow is a far more
difficult problem than resolving the Sommerfeld paradox. The boundary layer
effect adds tremendously to the difficulty. The Kolmogorov flow (periodic
boundary condition in every spatial dimension) is a much easier mathematical
problem in this aspect. The boundary layer is not present. Some progress on
the dynamical system studies on the Kolmogorov flow has been made \cite{MS61}
\cite{CP97} \cite{CP05} \cite{Li05} \cite{SYE06} \cite{LL08}. So far, proving
the existence of chaos in partial differential equations is only successful
for simpler systems, for a survey, see \cite{Li04}.

The article is organized as follows: In section 2, we will discuss some
possible phase spaces in which dynamical system studies can be conducted on
the Couette flow. In section 3, we prove the inviscid linear instability of
our sequence of oscillatory shears $U_{n}(y)$. In section 4, we prove the
viscous linear instability of our sequence of oscillatory shears $U_{n}(y)$
when viewed frozen. In section 5, we prove a bifurcation of our oscillatory
shears $U_{n}(y)$ to Kelvin's cat's eyes under the 2D Euler dynamics. In
section 6, we prove the inviscid linear instability of 3D shears $U(y,z)$ in a
neighborhood of our oscillatory shears $U_{n}(y)$.

\section{Mathematical Formulation}

We are interested in fluid flows between two infinite horizontal planes
(Figure \ref{Cflow}) where the upper plane moves with unit velocity and the
lower plate is fixed. The dynamics of such a fluid flow is governed by the
Navier-Stokes (NS) equations
\begin{equation}
\vec{u}_{t}+\vec{u}\cdot\nabla\vec{u}=-\nabla p+\epsilon\Delta\vec{u}%
\ ,\quad\nabla\cdot\vec{u}=0\ ; \label{NS}%
\end{equation}
defined in the spatial domain $D_{\infty}=\mathbb{R}\times\lbrack
0,1]\times\mathbb{R}$, where $\vec{u}=\left(  u_{1},u_{2},u_{3}\right)  $ is
the velocity, $p$ is the pressure, and $\epsilon$ is the inverse of the
Reynolds number $\epsilon=1/R$. The following boundary condition identifies
the specific flow
\begin{align}
&  u_{1}(t,x,0,z)=0,\quad u_{1}(t,x,1,z)=1,\nonumber\\
&  u_{i}(t,x,0,z)=u_{i}(t,x,1,z)=0,\ (i=2,3). \label{nsbc}%
\end{align}
\begin{figure}[th]
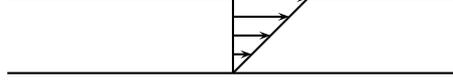

\vspace{0.5in} \centering
\[
\psline(-3,0)(3,0) \psline(-3,1)(3,1) \psline(0,0)(1,1) \psline(0,0)(0,1)
\psline{->}(0,0.5)(0.5,0.5) \psline{->}(0,1)(1,1)
\psline{->}(0,0.25)(0.25,0.25) \psline{->}(0,0.75)(0.75,0.75)
\]
\caption{Couette flow. }%
\label{Cflow}%
\end{figure}

\noindent The linear shear is given by
\begin{equation}
u_{1}=y,\quad u_{2}=u_{3}=0. \label{Cf}%
\end{equation}
One can choose the infinite dimensional phase space to be
\[
\hat{S}=\bigg \{u\ \bigg |\ u\in H_{\text{loc}}^{s}(D_{\infty}),\ (s\geq
3),\ \nabla\cdot\vec{u}=0,\text{ together with }(\ref{nsbc})\bigg \},
\]
where $H_{\text{loc}}^{s}$ is a local Sobolev space. This phase space $\hat
{S}$ is too large and too difficult to analyze. One can first study its
invariant subspace by posing extra periodic boundary condition along $x$ and
$z$ directions with periods $L_{1}$ and $L_{3}$. Denote by $D$ the partially
periodic domain $D=[0,L_{1}]\times\lbrack0,1]\times\lbrack0,L_{3}]$, the
partially periodic invariant subspace is given by
\begin{equation}
S=\bigg \{u\ \bigg |\ u\in H^{s}(D),\ (s\geq3),\ \nabla\cdot\vec{u}=0,\text{
together with }(\ref{nsbc})\bigg \}. \label{PhS}%
\end{equation}
Often we are interested in the two-dimensional reduction $u_{3}=\partial
_{3}=0$ in which case the phase space $S$ is simplified further. Moreover, in
the two-dimensional case, the flow is nicer since both 2D NS and 2D Euler
equations are globally well-posed.

\begin{remark}
By the change of variables
\[
u_{1}=y+v_{1},\ u_{2}=v_{2},\ u_{3}=v_{3},
\]
the new variable $\vec{v}=\left(  v_{1},v_{2},v_{3}\right)  $ satisfies the
Dirichlet boundary condition at $y=0,1$; thus
\[
\vec{v}=\sum_{n=1}^{+\infty}V_{n}(x,z)\sin n\pi y,
\]
and $V_{n}(x,z)$ is periodic in $x$ and $z$. By this representation, we see
that the phase space $S$ is a Banach manifold. In fact, the Navier-Stokes
equations (\ref{NS}) can be re-written in terms of the $v$ variable as
\begin{equation}
\vec{v}_{t}+\vec{v}\cdot\nabla\vec{v}=-\nabla p+\epsilon\Delta\vec{v}+\vec
{f}\ ,\quad\nabla\cdot\vec{v}=0\ ; \label{vNS}%
\end{equation}
where $\vec{f}=-y\partial_{x}\vec{v}-v_{2}\left(  1,0,0\right)  $. In terms of
$\vec{v}$, the phase spaces will be Banach spaces. \label{vRem}
\end{remark}

In our exploration inside the phase space, we will focus not only on the
neighborhood of the linear shear (\ref{Cf}), but also on the neighborhood of
the following sequence of oscillatory shears
\begin{equation}
u_{1}=U_{n}(y)=y+\frac{A}{n}\sin(4n\pi y),\ \left(  \frac{1}{2}\frac{1}{4\pi
}<A<\frac{1}{4\pi}\right)  ,\quad u_{2}=u_{3}=0, \label{Os}%
\end{equation}
which will be proved later to be linearly unstable under the 2D Euler flow.

All the 2D shears together form an invariant submanifold
\begin{equation}
\Lambda=\bigg \{u\in S\ \bigg |\ u_{1}=U(y),\ u_{2}=u_{3}%
=0,\ U(0)=0,\ U(1)=1\bigg \} \label{invS}%
\end{equation}
where $U(y)$ is an arbitrary function. Inside the invariant submanifold
$\Lambda$, the dynamics is governed by
\[
\partial_{t}U=\epsilon\partial_{y}^{2}U.
\]
The fixed point of this equation is given by
\[
\partial_{y}^{2}U=0,\quad\text{ i.e. }U=c_{1}y+c_{2}.
\]
The boundary conditions $U(0)=0$ and $U(1)=1$ imply that
\[
U=y
\]
which is the linear shear (\ref{Cf}). That is, when $\epsilon\neq0$, the
linear shear is the only fixed point inside $\Lambda$. Of course, when
$\epsilon=0$ (Euler flow), $\Lambda$ is an equilibrium manifold. Let
\[
U=y+V,
\]
then
\[
\partial_{t}V=\epsilon\partial_{y}^{2}V,\quad V(0)=V(1)=0.
\]
Thus,
\[
V=\sum_{n=1}^{+\infty}a_{n}e^{-\epsilon(n\pi)^{2}t}\sin n\pi y,
\]
where $a_{n}$'s are constants. Finally the orbits inside $\Lambda$ is given
by
\[
U=y+\sum_{n=1}^{+\infty}a_{n}e^{-\epsilon(n\pi)^{2}t}\sin n\pi y.
\]

\section{Inviscid Linear Instability of the Sequence of Oscillatory Shears}

Let ($U(y), 0$) be a steady shear of the 2D Navier-Stokes or Euler equation,
e.g. $U(y)=y$. Introducing the stream function $\psi$:
\[
u_{1}=\frac{\partial\psi}{\partial y},\quad u_{2}=-\frac{\partial\psi
}{\partial x},
\]
and linearizing at the steady shear in the form
\[
\psi=\phi(y)e^{i\alpha x+\lambda t}=\phi(y)e^{i\alpha(x-ct)},\quad
\lambda=-i\alpha c;
\]
one obtains the so-called Orr-Sommerfeld equation
\begin{equation}
\frac{\epsilon}{i\alpha}[\partial_{y}^{2}-\alpha^{2}]^{2}\phi+U^{\prime\prime
}\phi-(U-c)[\partial_{y}^{2}-\alpha^{2}]\phi=0, \label{Orr-S}%
\end{equation}
with the boundary conditions
\begin{equation}
\phi=\phi^{\prime}=0,\text{ at }y=0,\ 1. \label{bc-OS}%
\end{equation}
For Euler equation, $\epsilon=0$ and (\ref{Orr-S}) is reduced to the Rayleigh
equation
\begin{equation}
U^{\prime\prime}\phi-(U-c)[\partial_{y}^{2}-\alpha^{2}]\phi=0,
\label{Rayleigh}%
\end{equation}
with the boundary conditions
\begin{equation}
\phi=0,\text{ at }y=0,\ 1. \label{bc-Rayleigh}%
\end{equation}
We call $\left(  \alpha,c\right)  $ with $\alpha>0,c\in\mathbf{C\ }$an
eigenmode of Orr-Sommerfeld or Rayleigh equations, if for such $\left(
\alpha,c\right)  $ the equation (\ref{Orr-S}) or (\ref{Rayleigh}) is solvable
and the corresponding solution is called an eigenfunction. The eigenmode with
$\operatorname{Im}c>0$ is called unstable, with $\operatorname{Im}c=0$ is
called neutral. As shown in last section, all the shears ($U(y),0$) are steady
under the 2D Euler dynamics; while only the linear shear $U(y)=y$ is steady
under the 2D Navier-Stokes dynamics. Nevertheless, when the Reynolds number
$R$ is large ($\epsilon$ is small), the shears ($U(y),0$) only drift slowly.
When they are viewed frozen (or artificial forces make them steady), the
spectra of the corresponding Orr-Sommerfeld operator are still significant in
predicting their transient instabilities.

First indicated by Sommerfeld \cite{Som08}, the linear shear $U\left(
y\right)  =y\ $is linearly stable for all Reynolds number. This fact was
rigorously proved by Romanov \cite{Rom73} who showed that the eigenvalues of
the linearized Navier-Stokes Operator satisfies
\[
\operatorname{Re}\lambda<-C/R;
\]
and moreover, the linear shear is nonlinearly stable for all Reynolds number
in $W^{1,2}([0,1]\times\mathbb{R}^{2})$. When $\epsilon=0$, for the linear
shear, the Rayleigh equation reduces to
\[
(U-c ) \varphi= 0 , \text{ where } \varphi= [\partial_{y}^{2} - \alpha^{2}]
\phi,
\]
which has only continuous spectrum \cite{Fad71}
\[
c \in\left[  \min U(y) , \max U(y) \right]  = [0,1] ,
\]
and the corresponding quasi-eigenfunction for $c = U(y_{0})$ is given by
\[
\varphi= \delta(y- y_{0}).
\]
In the $\phi$ variable
\[
[\partial_{y}^{2} - \alpha^{2}] \phi= \delta(y- y_{0}) ,
\]
that is, $\phi$ is the Green function. In fact, the linear shear is also
nonlinearly stable under the 2D Euler flow. In the vorticity form, the 2D
Euler equation is given by
\[
\frac{\partial\Omega}{\partial t}+u\cdot\nabla\Omega=0,\quad\nabla\cdot u=0;
\]
which has the invariants $\int F(\Omega)dxdy$ for any $F$. Near the linear
shear, $\Omega=1+\omega$, and $\omega$ satisfies
\[
\frac{\partial\omega}{\partial t}+u\cdot\nabla\omega=0,\quad\nabla\cdot u=0;
\]
where $u$ is the velocity corresponding to $\Omega$. Then
\[
\int F(\omega)dxdy
\]
are also invariant for any $F$. Thus the linear shear is nonlinearly stable in
$L^{p}$ norm of vorticity for any $p\in\left[  1,+\infty\right]  $.

When $\epsilon=0$, each point in $\Lambda$ is a fixed point. The linear
spectrum of these shear flows has been studied a lot since Lord Rayleigh in
1880s \cite{ray}. If the profile of the shear does not contain any inflection
point, then by Rayleigh's criterion there is no unstable eigenvalue and the
spectrum consists of only continuous spectrum given by the imaginary axis.
However, the existence of an inflection point is only necessary for linear
instability and the results on sufficient conditions for instability remain
very limited \cite{drazin-reid} \cite{Lin03}. To construct unstable shears
near the linear shear, we use the following instability criterion for monotone
shear flows.

\begin{lemma}
\label{lemma-monotone-shear}Consider a monotone shear profile $U(y)\in
C^{2}\left(  0,1\right)  $ with inflection points$.$ Let $y^{0},\cdots,y^{l}$
be all the inflection points and $\left\{  U^{i}=U\left(  y^{i}\right)
\right\}  _{i=0}^{l}$ be the inflection values. Define%
\begin{equation}
Q_{i}(y)=\frac{U^{\prime\prime}(y)}{U(y)-U^{i}} \label{Po}%
\end{equation}
and the Sturm-Liouville operator
\begin{equation}
L_{i}\varphi=-\varphi^{\prime\prime}+Q_{i}(y)\varphi\label{SLO}%
\end{equation}
with the Dirichlet boundary condition $\varphi(0)=\varphi(1)=0$. (i) If for
some $0\leq i\leq l$, the operator $L_{i}$ has a negative eigenvalue, then the
Rayleigh equation (\ref{Rayleigh}) has unstable eigenmodes for some intervals
of wave numbers. Any end point $\alpha_{s}$ of the unstable intervals is such
that $-\alpha_{s}^{2}$ being a negative eigenvalue of some operator $L_{i}.$
(ii) If $L_{i}\geq0$ for all $0\leq i\leq l$, then the shear profile $U\left(
y\right)  $ is linearly stable for any wave number.
\end{lemma}

Lemma \ref{lemma-monotone-shear} was rigorously proved in \cite{Lin05}. The
following key observation is due to Tollmien \cite{Tol35} in 1930 (see also
\cite{Fad71} and \cite{Lin45}): if $-\alpha_{s}^{2}$ is a negative eigenvalue
of $L_{i}$ and $\phi_{s}$ is the eigenfunction, then $\left(  c,\alpha\right)
=\left(  U^{i},\alpha_{s}\right)  $ is a neutral eigenmode to the Rayleigh
equation (\ref{Rayleigh}), with $\phi_{s}$ being the eigenfunction; then one
can try to find unstable modes $\left(  c,\alpha,\phi\right)  $ $\left(
\operatorname{Im}c>0\right)  \ $to (\ref{Rayleigh}) near such neutral mode. A
variational formula for $\partial c/ \partial\alpha$ can be derived at the
neutral mode.

Notice that any point in $\Lambda$ has the representation
\begin{equation}
U\left(  y\right)  =y+\sum_{n=1}^{+\infty}a_{n}\sin n\pi y, \label{foexp}%
\end{equation}
so the sequence of oscillatory shears (\ref{formula-U_n}) can be viewed as
single modes of the above expansion. The following theorem shows the inviscid
linear instability of the oscillatory shears (\ref{formula-U_n}). Since
$\left\vert U_{n}\left(  y\right)  -y\right\vert _{L^{\infty}\left(
0,1\right)  }\leq\frac{A}{n}$, this shows that there exist unstable shears in
any small $L^{\infty}$ (velocity) neighborhood of the linear shear (\ref{Cf}).
In particular, these unstable shears are arbitrarily small kinetic energy
perturbations to the linear shear.

\begin{theorem}
Under the 2D Euler dynamics, the oscillatory shears $U_{n}\ $defined by
(\ref{formula-U_n}) are linearly unstable. More precisely, there exists an
unstable eigenmode curve $\left(  \alpha,c(\alpha)\right)  $ with
$\operatorname{Im}c(\alpha)>0$ of the Rayleigh equation (\ref{Rayleigh}) with
$U_{n}$, stemming from a neutral mode ($\alpha_{n},1/2$) where $\alpha_{n}\geq
c_{0}n$ ($c_{0}>0$ is independent of $n$). The corresponding unstable
eigenfunctions are in $C^{\infty}\left(  0,1\right)  $. \label{OST}
\end{theorem}

\begin{proof}
First notice that the oscillatory shears (\ref{formula-U_n}) are monotone,
since $U_{n}^{\prime}>0$. So to show the linear instability of $U_{n}$, by
Lemma \ref{lemma-monotone-shear} it suffices to prove that for at least one
inflection point, the Sturm-Liouville operator (\ref{SLO}) has a negative
eigenvalue. We choose the inflection point of $U_{n}$ at $y=1/2$. Define
\begin{equation}
Q(y)=\frac{U^{\prime\prime}(y)}{U(y)-U(\frac{1}{2})}=\frac{-16\pi^{2}%
nA\sin(4n\pi y)}{y-\frac{1}{2}+\frac{A}{n}\sin(4n\pi y)}\ \label{ExQ}%
\end{equation}
and the Sturm-Liouville operator $L=-\frac{d^{2}}{dy^{2}}+Q(y)$. By
Rayleigh-Riesz principle, the smallest eigenvalue $\lambda_{1}\ $of $L$ is
given by
\begin{equation}
\lambda_{1}=\min_{\varphi\in H_{0}^{1}\left(  0,1\right)  }\frac{\int_{0}%
^{1}\left[  \left(  \varphi^{\prime}\right)  ^{2}+Q\varphi^{2}\right]
dy}{\int_{0}^{1}\varphi^{2}dy}\ . \label{riesz-Rayleigh}%
\end{equation}
Thus, to prove that $\lambda_{1}$ is negative, all we need to do is to show
that the fraction on the right hand side of above is negative for some
specific test function $\varphi\in H_{0}^{1}\left(  0,1\right)  $. Let
$y^{\prime}=y-\frac{1}{2}$, $y^{\prime}\in\left[  -\frac{1}{2},\frac{1}%
{2}\right]  \ $and denote $\tilde{Q}(y^{\prime})=Q(y^{\prime}+\frac{1}{2})$.
When $y^{\prime}\in\lbrack-\frac{1}{8n},\frac{1}{8n}],$ we have
\[
\frac{2}{\pi}|4n\pi y^{\prime}|\leq|\sin(4n\pi y^{\prime})|\leq|4n\pi
y^{\prime}|,
\]
thus
\[
\tilde{Q}(y^{\prime})=\frac{-16\pi^{2}nA\left\vert \sin(4n\pi y^{\prime
})\right\vert }{\left\vert y^{\prime}\right\vert +\frac{A}{n}\left\vert
\sin(4n\pi y^{\prime})\right\vert }\leq\frac{-16\pi^{2}nA\frac{2}{\pi
}\left\vert 4n\pi y^{\prime}\right\vert }{\left\vert y^{\prime}\right\vert
+\frac{A}{n}\left\vert 4n\pi y^{\prime}\right\vert }=\frac{-2(8\pi)^{2}n^{2}%
A}{1+4\pi A}\ .
\]
We choose the test function $\varphi$ as follows (Figure \ref{tf}):
\[
\varphi(y^{\prime})=\left\{
\begin{array}
[c]{l}%
\frac{1}{4n},\quad-\frac{1}{8n}\leq y^{\prime}\leq\frac{1}{8n};\cr\frac{3}%
{8n}-y^{\prime},\quad\frac{1}{8n}<y^{\prime}\leq\frac{3}{8n};\cr\frac{3}%
{8n}+y^{\prime},\quad-\frac{3}{8n}\leq y^{\prime}<-\frac{1}{8n};\cr0,\quad
\frac{1}{2}\geq|y^{\prime}|>\frac{3}{8n}.\cr
\end{array}
\right.
\]
\begin{figure}[th]
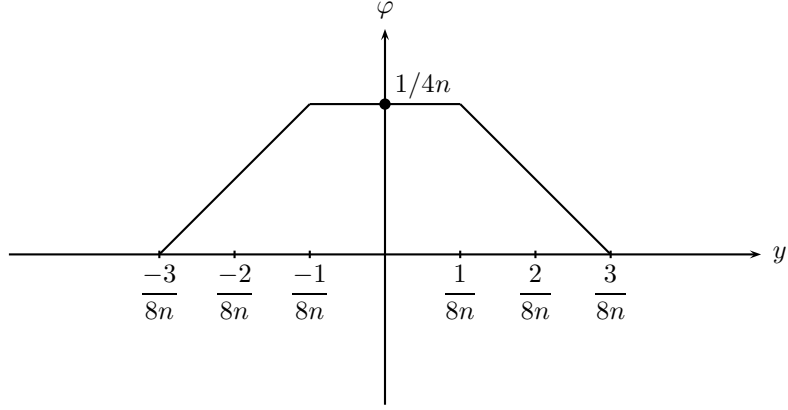

\centering
\vspace{1.0in}
\[
\psline{->}(-5,0)(5,0) \rput(5.25,0){y} \psline{->}(0,-2)(0,3)
\rput(0,3.25){\varphi} \psline(1,-0.05)(1,0.05) \psline(2,-0.05)(2,0.05)
\psline(3,-0.05)(3,0.05) \psline(-1,-0.05)(-1,0.05) \psline(-2,-0.05)(-2,0.05)
\psline(-3,-0.05)(-3,0.05) \rput(1,-0.5){\frac{1}{8n}} \rput(2,-0.5){\frac
{2}{8n}} \rput(3,-0.5){\frac{3}{8n}} \rput(-1,-0.5){\frac{-1}{8n}}
\rput(-2,-0.5){\frac{-2}{8n}} \rput(-3,-0.5){\frac{-3}{8n}} \psline(-1,2)(1,2)
\rput(0.5,2.25){1/4n} \pscircle[fillcolor=black,fillstyle=solid](0,2){.075}
\psline(1,2)(3,0) \psline(-1,2)(-3,0)
\]
\vspace{0.5in}\caption{The test function $\varphi$.}%
\label{tf}%
\end{figure}Since both $\tilde{Q}(y^{\prime})$ and $\varphi(y^{\prime})$ are
even in $y^{\prime}$, we only need to estimate over the interval $\left[
0,\frac{3}{8n}\right]  $. First we have
\[
\int_{0}^{\frac{3}{8n}}\left(  \varphi^{\prime}\right)  ^{2}dy^{\prime}%
=\frac{1}{4n}.
\]
Notice that for $\eta\in\lbrack0,\frac{1}{8n}]$,
\[
U_{n}(\frac{1}{4n}+\eta+\frac{1}{2})-U_{n}(\frac{1}{2})>U_{n}(\frac{1}%
{4n}-\eta+\frac{1}{2})-U_{n}(\frac{1}{2})>0,
\]
since $U_{n}$ is monotonically increasing, and
\[
-\sin\left[  4n\pi\left(  \frac{1}{4n}+\eta\right)  \right]  =\sin\left[
4n\pi\left(  \frac{1}{4n}-\eta\right)  \right]  .
\]
Thus
\[
-\tilde{Q}(\frac{1}{4n}-\eta)>\tilde{Q}(y^{\prime}=\frac{1}{4n}+\eta)\geq0.
\]
Notice also that
\[
\varphi(\frac{1}{4n}-\eta)>\varphi(\frac{1}{4n}+\eta)\geq0,
\]
so
\[
\int_{\frac{1}{8n}}^{\frac{3}{8n}}\tilde{Q}\varphi^{2}dy^{\prime}=\int
_{0}^{\frac{1}{8n}}\left[  \tilde{Q}(\frac{1}{4n}-\eta)\varphi^{2}(\frac
{1}{4n}-\eta)+\tilde{Q}(\frac{1}{4n}+\eta)\varphi^{2}(\frac{1}{4n}%
+\eta)\right]  d\eta\leq0.
\]
On the other hand,
\begin{align*}
-\int_{0}^{\frac{1}{8n}}\tilde{Q}\varphi^{2}dy^{\prime}  &  \geq\frac
{8A\pi^{2}(4n)^{2}}{1+4A\pi}\left(  \frac{1}{4n}\right)  ^{2}\frac{1}{8n}\\
&  =\frac{4A\pi^{2}}{1+4A\pi}\frac{1}{4n}=\frac{\pi(1-\delta)}{2-\delta}%
\frac{1}{4n},
\end{align*}
where $4A\pi=1-\delta$, $0<\delta<1/2$ for the oscillatory shears
(\ref{formula-U_n}). We also have
\[
\int_{0}^{\frac{3}{8n}}\varphi^{2}dy^{\prime}=\left(  \frac{1}{4n}\right)
^{2}\frac{1}{8n}+\frac{1}{3}\left(  \frac{1}{4n}\right)  ^{3}=\frac{5}%
{6}\left(  \frac{1}{4n}\right)  ^{3}.
\]
Gathering all the above estimates together, we get
\[
\frac{\int_{0}^{1}\left[  \left(  \varphi^{\prime}\right)  ^{2}+Q\varphi
^{2}\right]  dy}{\int_{0}^{1}\varphi^{2}dy}=\frac{\int_{0}^{\frac{3}{8n}%
}\left[  \left(  \varphi^{\prime}\right)  ^{2}+\tilde{Q}\varphi^{2}\right]
dy^{\prime}}{\int_{0}^{\frac{3}{8n}}\varphi^{2}dy^{\prime}}\leq-\frac{6}%
{5}(4n)^{2}\left[  \frac{\pi(1-\delta)}{2-\delta}-1\right]  .
\]
By (\ref{riesz-Rayleigh}), this means
\begin{equation}
\lambda_{1}\leq-\frac{6}{5}(4n)^{2}\left[  \frac{\pi(1-\delta)}{2-\delta
}-1\right]  <0 \label{lambda-1}%
\end{equation}
when $0<\delta<\frac{\pi-2}{\pi-1}\approx0.533$. Since $0<\delta<1/2$ for the
oscillatory shears $U_{n}\left(  y\right)  \ $(\ref{Os}), the corresponding
$\lambda_{1}$ is negative. Thus by Lemma \ref{lemma-monotone-shear}, we get
unstable modes $\left(  \alpha,c\left(  \alpha\right)  \right)  $ $\left(
\operatorname{Im}c\left(  \alpha\right)  >0\right)  $ to the Rayleigh equation
(\ref{Rayleigh}) associated with $U_{n}\left(  y\right)  ,$ for certain
intervals of wave number $\alpha$. In particular, the unstable wave number
intervals include a neighborhood of $\alpha_{n}=\sqrt{-\lambda_{1}}\geq
c_{0}n$ and for those wave numbers $\alpha$, the unstable eigenvalue $c\left(
\alpha\right)  $ is close to the inflection value $\frac{1}{2}$. Here, $c_{0}$
is a positive constant independent of $n$ (see (\ref{lambda-1})). Since
$U_{n}\left(  y\right)  \in C^{\infty}\left(  0,1\right)  $, the regularity of
the unstable eigenfunction follows easily from the Rayleigh equation
(\ref{Rayleigh}). This completes the proof of the theorem.
\end{proof}

\begin{remark}
\label{remark-inviscid-insta}1). Even though they approach the Couette flow
(\ref{Cf}) in the energy norm and $L^{\infty}$ norm of velocity as
$n\rightarrow\infty$, the oscillatory shears (\ref{Os}) do not approach the
linear shear in $H^{s}$ ($s\geq1$) as $n\rightarrow\infty$; by simply noticing
that
\[
U_{n}^{\prime}=1+4\pi A\cos(4n\pi y).
\]
Thus in our phase space $S$, the oscillatory shears are of a finite distance
away from the linear shear for large $n$. On the other hand, in fluid
experiments, it is the $L^{\infty}$ norm of velocity that is observed; thus
for large $n$, the oscillatory shears (\ref{Os}) can be considered to be small
noises by fluid experimentalists.

2). The form of the unstable shears $U_{n}\left(  y\right)  $
(\ref{formula-U_n}) is not unique. For example, we could choose
\begin{equation}
U_{n}\left(  y\right)  = y +\frac{A}{n}W\left(  ny\right)
\label{formula-general-U_n}%
\end{equation}
where $W\left(  y\right)  \in C^{2}\left(  0,1\right)  $ is a \thinspace
$1-$periodic function of the similar shape as $\sin y$. By choosing constant
$A$ properly, the proof of Theorem \ref{OST} can still go through to get the
linear instability of $U_{n}\left(  y\right)  $. Nevertheless, we consider the
oscillatory shears (\ref{Os}) as good representatives of linearly unstable
shears near the linear shear since they are single Fourier modes of
(\ref{foexp}).

We also notice that the oscillatory structure of $U_{n}$ (\ref{formula-U_n})
is in some sense necessary for instability. First, for any given function
$f\left(  y\right)  \in C^{2}\left(  0,1\right)  $, the shear flow
$U_{\varepsilon}\left(  y\right)  =y+\varepsilon f\left(  y\right)  $ is
spectrally stable, i.e., there is no unstable solution to the Rayleigh
equation (\ref{Rayleigh}), when $\varepsilon$ is small enough. This can be
seen as follows. When $\varepsilon$ is small, $U_{\varepsilon}\left(
y\right)  $ is monotone. So by Lemma \ref{lemma-monotone-shear} (ii),
$U_{\varepsilon}\left(  y\right)  $ is spectrally stable if all the operators
$L_{i}=-\frac{d^{2}}{dy^{2}}+Q_{i}(y)$ are nonnegative, where%
\[
Q_{i}(y)=\varepsilon\frac{f^{\prime\prime}\left(  y\right)  }{U_{\varepsilon
}\left(  y\right)  -U_{\varepsilon}\left(  y_{i}\right)  }%
\]
and $y_{i},i=1,\cdots l,$ are all the inflection points of $f\left(  y\right)
$. When $\varepsilon$ is small, $\left\vert Q_{i}(y)\right\vert _{L^{\infty}}$
is small so $\mathcal{L}_{i}\geq0$ which proves the linear stability of
$U_{\varepsilon}\left(  y\right)  .$ Second, for shears of the form
(\ref{formula-general-U_n}) with $W\left(  y\right)  \in C^{2}\left(
\mathbf{R}\right)  $, to preserve the Couette boundary conditions (\ref{nsbc})
for all large $n$, the function $W\left(  y\right)  $ must have infinitely
many zeros. So it is natural to choose $W$ as a periodic function with zeros.

3). For the Sturm-Liouville operator $L=-\frac{d^{2}}{dy^{2}}+Q(y)$ with
$Q(y)$ defined by (\ref{ExQ}), the second eigenvalue $\lambda_{2}\geq0$ (see
Appendix for a proof). Thus $\lambda_{1}$ is the only negative eigenvalue of
$L$, so by Sturm-Liouville theory $\lambda_{1}$ is simple and the
corresponding eigenfunction $\phi_{n}\left(  y\right)  $ can be chosen such
that $\phi_{n}\left(  y\right)  >0$ when $y\in\left(  0,1\right)  $. This fact
will be used later in the proof of Theorem \ref{thm-catseye}.
\end{remark}

Let $D=[0,L_{1}]\times\lbrack0,1]$, where $L_{1}=\frac{2\pi}{\alpha}$ and
$\alpha$ is an unstable wave number for $U_{n}\left(  y\right)  $ with $n$
large. As a corollary of Theorem \ref{OST}, we have

\begin{theorem}
The spectra of the linear Euler operator at the oscillatory shear (\ref{Os})
in $H^{s}(D)$ ($s \geq3$) are as follows:

\begin{enumerate}
\item There are $J \geq1$ unstable eigenvalues and $J \geq1$ stable eigenvalues.

\item The imaginary axis is the absolutely continuous spectrum.
\end{enumerate}
\end{theorem}

The first claim is proved in Theorem \ref{OST}, and the proof of the second
claim can be found in \cite{Fad71}.

For any linearly unstable shear flow of Euler equation, the nonlinear
instability can be established too (i.e. \cite{bgs}, \cite{gre00},
\cite{Lin04}). Moreover, the existence of unstable and stable manifolds near
unstable shear flows was recently proved \cite{lin-zeng} for Euler equation.

\section{Viscous Linear Instability of the Sequence of Oscillatory Shears}

A more precise statement of the Sommerfeld paradox is as follows:

\begin{itemize}
\item Mathematically, the linear shear is linearly and nonlinear stable for
all Reynolds number $R$, in fact, all the eigenvalues of the Orr-Sommerfeld
operator satisfy the bound $\lambda< -C / R$ where $C$ is a positive constant
\cite{Rom73}.

\item Experimentally, for any $R>360$ (where $R=\frac{1}{4\epsilon}$ in our
setting \cite{BTAA95}), there exists a threshold amplitude of perturbations,
of order $\O (R^{-\mu})$ where $1\leq\mu<\frac{21}{4}$ depends on the type of
the perturbations \cite{KLH94}, which leads to transition to turbulence.
\end{itemize}

A mathematically more precise re-statement of this experimental claim is as
follows: For any fixed amplitude of perturbations to the linear shear, when
$R$ is sufficiently large, transition to turbulence occurs. For any fixed $R$,
when the amplitude of perturbations is sufficiently large, transition to
turbulence occurs. There may even be an asymptotic relation between such
amplitude threshold and $R$.

Our main idea of the resolution is as follows: The oscillatory shears
(\ref{Os}) are perturbations of the linear shear. As $n\rightarrow\infty$,
they approach the linear shear in $L^{\infty}$ norm of velocity. They are
linearly unstable under the 2D Euler dynamics. As shown later on, this will
lead to the existence of an unstable solution of the Orr-Sommerfeld equation
(\ref{Orr-S}) with $U$ given by the oscillatory shears (\ref{Os}), when the
Reynolds number $R$ is sufficiently large. Notice that these oscillatory
shears are not fixed points anymore under the Navier-Stokes dynamics.
Nevertheless, they only drift very slowly. The important fact is that here the
unstable eigenvalue of the Orr-Sommerfeld operator is order $\O (1)$ with
respect to $\epsilon=1/R$ as $\epsilon\rightarrow0^{+}$. This fact should lead
to a transient nonlinear growth near the oscillatory shears (and the linear
shear) which manifests as a transition to turbulence. This has been confirmed
numerically \cite{Lan09}. Here the amplitude of the perturbation from the
linear shear will be measured by the deviation of the oscillatory shears from
the linear shear and the perturbation on top of the oscillatory shears. One
final note is that here the turbulence is often transient, i.e. with a finite
life time after which the flow re-laminates back to the linear shear.

As mentioned in the Introduction, due to the recent discovery of 3D steady
states, there has been a conjecture on the explanation of the Sommerfeld
paradox using the stable manifolds of the 3D steady states. The stable
manifolds can not be connected to the linear shear since it is linearly
stable. The conjecture is that the stable manifolds can get close to the
linear shear, which leads to the transition from the linear shear to
turbulence \cite{GHC08}. The main criticism on this argument is that these 3D
steady states are quite far away from the linear shear even in the velocity
variable. For each of such 3D steady states, its stable manifold will have a
fixed distance from the linear shear due to its Romanov's nonlinear stability
theorem \cite{Rom73}. When the Reynolds number $R$ is large enough, this
distance will be much bigger than the threshold $\O (R^{-\mu})$ of
perturbations given above. This contradiction shows that the stable manifold
explanation is not so satisfactory.

Based upon the fact that the oscillatory shears (\ref{Os}) are linearly
unstable under the 2D Euler dynamics, it is natural to search for an unstable
solution of the Orr-Sommerfeld equation (\ref{Orr-S}) with $U$ given by
(\ref{formula-U_n}). The main difficulty in this search naturally lies at the
boundary conditions. For Rayleigh equation (\ref{Rayleigh}), only two boundary
conditions are necessary while four are required for Orr-Sommerfeld. When
$\epsilon$ is small, of course the other two boundary conditions will create
boundary layer effects. By using the asymptotic expansion theory developed by
W. Wasow \cite{Was48} 60 years ago, we can prove the following.

\begin{theorem}
Let ($\alpha^{0},c^{0}$) be on the unstable eigenmode curve of the Rayleigh
equation associated with the oscillatory shears $U_{n}$ (\ref{formula-U_n}),
which stems from the neutral mode ($\alpha_{0},1/2$). Then for $\epsilon$
sufficiently small, there exists an unstable eigenmode ($\alpha^{0},c^{\ast}$)
with $\operatorname{Im}c^{\ast}>0$ of the Orr-Sommerfeld equation
(\ref{Orr-S}) with $U$ given by (\ref{formula-U_n}). When $\epsilon
\rightarrow0^{+}$, $c^{\ast}\rightarrow c^{0}$. \label{LOS}
\end{theorem}

\begin{proof}
Rewrite the Orr-Sommerfeld equation (\ref{Orr-S}) in the following form:
\begin{equation}
\left[  \partial_{y}^{2}-\alpha^{2}\right]  ^{2}\phi+\gamma^{2}\left[
b(y)\left(  \partial_{y}^{2}-\alpha^{2}\right)  +U_{n}^{\prime\prime}\right]
\phi=0, \label{rwos}%
\end{equation}
where $U_{n}$ is given by the oscillatory shears (\ref{Os}),
\[
\gamma^{2}=i\alpha R,\quad b(y)=-(U_{n}-c).
\]
Let ($\alpha^{0},c^{0}$) be on the unstable eigenvalue curve of the linear 2D
Euler operator ($\epsilon=0$ in (\ref{Orr-S})), stemming from the neutral mode
($\alpha_{n},1/2$) as identified by Theorem \ref{OST}, $c^{0}=c_{r}^{0}%
+ic_{i}^{0}$, $c_{i}^{0}>0$ ($c_{r}^{0}$ is near $1/2$). We will study the
segment where $c^{0}$ is sufficiently close to $1/2$, and the region $c\in
B_{\delta_{1}}(c^{0})$ --- a circular disc of radius $0<\delta_{1}<\frac{1}%
{2}c_{i}^{0}$ and centered at $c^{0}$. See Figure \ref{alph}.

\begin{figure}[th]
\vspace{0.5in} \begin{tikzpicture}[scale=1.5]
\draw (2,0) -- (2,4) (0,2) -- (4,2);
\draw (2,3.75) -- (2.25,3.75) -- (2.25,4);
\draw (3,2) .. controls (3.0625,2.25) and (3,3) .. (2.125,3);
\draw (2.75,2.8125) circle (.125) node [right=4pt] {$B_{\delta_1(c^0)}$};
\draw[fill] (2.75,2.8125) circle (.03125);
\draw[fill] (3,2) circle (.03125) node [below] {$\frac{1}{2}$};
\draw (2.125,3.875) node {$c$};
\end{tikzpicture}
\caption{A $c$-plane illustration.}%
\label{alph}%
\end{figure}
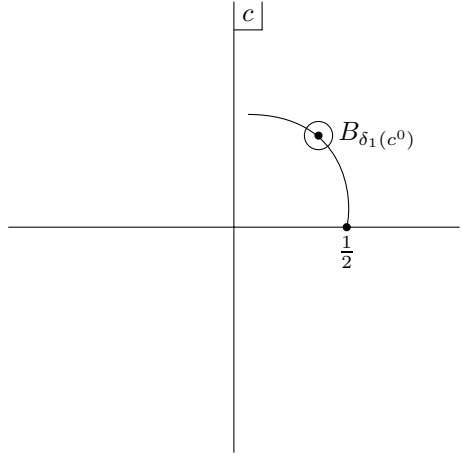\begin{figure}[th]
\vspace{0.5in} \begin{tikzpicture}[scale=2]
\draw (1,0) -- (1,4);
\draw (0,2) -- (4,2);
\draw (2.125,2.125) -- (2.125,1.25) node [below] {$C_3$};
\draw (2.125,2.125) -- (1.3125,2.625) node [above] {$C_2$};
\draw (2.125,2.125) -- (2.9325,2.625) node [above] {$C_1$};
\draw (2.125,2) ellipse (1.5 and .75);
\draw (2.125,2.125) node [above=6pt] {$S_3$};
\draw (1.375,2.25) node {$S_1$};
\draw (2.75,2.25) node {$S_2$};
\draw[fill] (2,2) circle (.03125) node[below] {$\frac{1}{2}$};
\draw[fill] (3,2) circle (.03125) node[below] {$1$};
\draw[fill] (1,2) circle (.03125) node[below right] {$0$};
\draw[fill] (2.125,2.125) circle (.03125) node[right=2pt] {$\hat{y}$};
\draw (1.125,3.875) node {$y$};
\draw (1,3.75) -- (1.25,3.75) -- (1.25,4);
\draw (2.5,1) node {$S$};
\end{tikzpicture}
\caption{A $y$-plane illustration.}%
\label{rods}%
\end{figure}
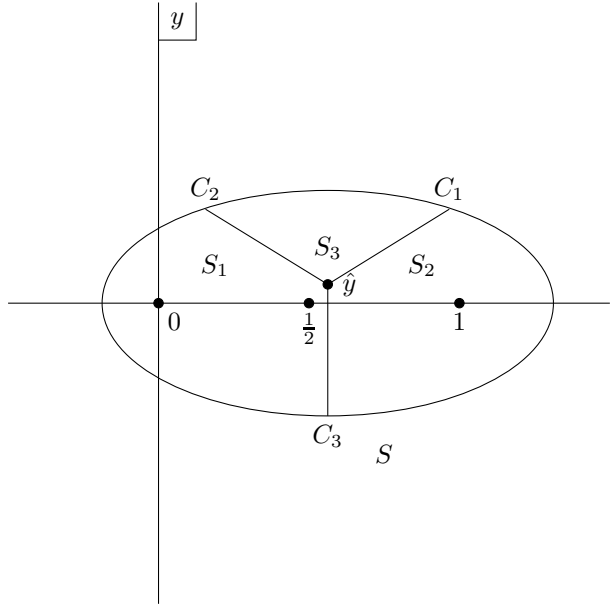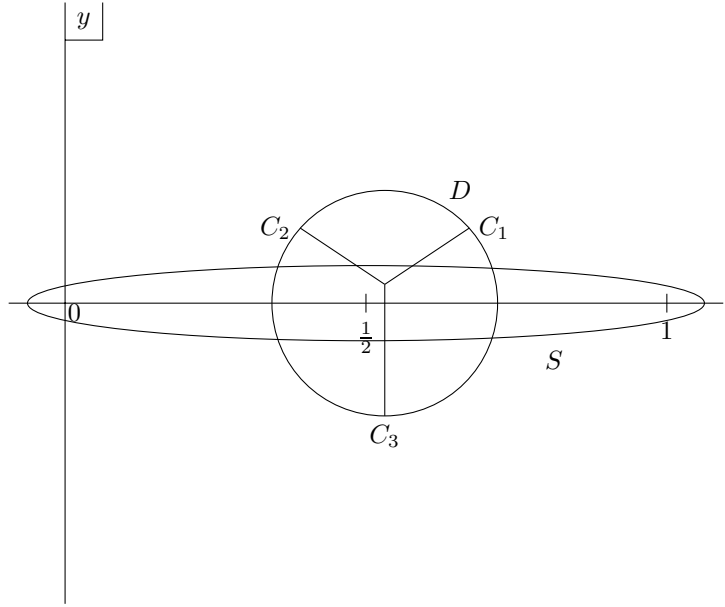
\begin{figure}[th]
\vspace{0.5in} \begin{tikzpicture}[scale=2]
\draw (0,0) -- (0,4) (-.375,2) -- (4.375,2) (0,3.75) -- (.25,3.75) -- (.25,4);
\draw (2,2) ellipse (2.25 and .25);
\draw (2.125,2) circle (.75);
\draw (2.125,2.125) -- (2.125,1.25) node [below] {$C_3$};
\draw (2.125,2.125) -- (2.6875,2.5) node [right] {$C_1$};
\draw (2.125,2.125) -- (1.5625,2.5) node [left] {$C_2$};
\draw (2,2.0625) -- (2,1.9375) node [below] {$\frac{1}{2}$};
\draw (4,2.0625) -- (4,1.9375) node [below] {$1$};
\draw (.125,3.875) node {$y$};
\draw (3.25,1.625) node {$S$};
\draw (2.625,2.75) node {$D$};
\draw (.0625,1.9375) node {$0$};
\end{tikzpicture}
\caption{Another $y$-plane illustration.}%
\label{rod2}%
\end{figure}

\noindent We know that $U_{n}(y)-1/2$ has only one zero $y=1/2$ on the real
segment $y\in\lbrack0,1]$. Thus we can find a rod region $S$ around $[0,1]$,
in which $U_{n}(y)-1/2$ has only one zero at $y=1/2$. See Figure \ref{rods}.
Then by Rouch\'{e} theorem, when $c$ is sufficiently close to $1/2$,
$U_{n}(y)-c$ also has the only zero $\hat{y}$ near $1/2$:
\[
\hat{y}=\frac{1}{2}+\frac{1}{1+4\pi A}\left(  c-\frac{1}{2} \right)
+\O \left(  c-\frac{1}{2} \right)  ^{2},
\]
and
\begin{equation}
U_{n}^{\prime}(\hat{y})=(1+4\pi A)+\O \left(  c-\frac{1}{2} \right)  ^{2}%
\neq0. \label{deriv}%
\end{equation}
Define
\[
Q(y)=\int_{\hat{y}}^{y}\sqrt{-b(s)}ds.
\]
As $y$ circles around $\hat{y}$ once in counter-clockwise direction,
$\arg(\gamma Q(y))$ increases by $3\pi$. So there are three curves $C_{j}$
($j=1,2,3$) stemming from $\hat{y}$, on which $Re(\gamma Q(y))=0$. By the
derivative formula (\ref{deriv}), there is a disc $D$ centered at $y=1/2$, the
radius of which is independent of $c$, such that inside the disc, the three
curves $C_{j}$ can be well approximated by using
\[
U_{n}-c\sim(1+4\pi A)(y-\hat{y}).
\]
We can choose the rod region $S$ so thin that it penetrates $D$ as shown in
Figure \ref{rod2}. Direct calculation reveals that $C_{1}$ has an angle of
$\pi/6$, and the angles between neighboring $C_{j}$'s are $2\pi/3$. When
$c=1/2$, along the two pieces of the interval [$0,1$] outside the disc $D$,
$\arg(\gamma Q(y))$ is not changing. Thus when $c$ is sufficiently close to
$1/2$ but not equal to $1/2$, the entire real interval $[0,1]$ lies inside the
two sectors $S_{1}\cup S_{2}$ ($\hat{y}$ not included, Figure \ref{rods}).
Inside the two sectors $S_{1}\cup S_{2}$, Wasow obtained the following results
\cite{Was48}: There are four linearly independent solutions to the
Orr-Sommerfeld equation (\ref{rwos}) in the forms,
\begin{align}
\phi_{j}  &  =e^{\gamma Q(y)}\left[  \sum_{n=0}^{N}\psi_{jn}(y)\gamma
^{-n}+f(S_{1}\cup S_{2})\gamma^{-N-1}\right]  ,\quad(j=1,2);\label{asol1}\\
\phi_{\ell}  &  =\psi_{\ell}(y)+f(S_{1}\cup S_{2})\gamma^{-2},\quad(\ell=3,4);
\label{asol2}%
\end{align}
where $Re(\gamma Q(y))<0$ in $S_{j}$ for $\phi_{j}$, and changes sign when $y$
crosses into the other sector; $\psi_{jn}$ are analytic in $S_{1}\cup S_{2}$,
$N$ is a large number, $f(S_{1}\cup S_{2})$ denotes any function which,
together with all its derivatives in $y$, is uniformly bounded in $\gamma$ in
every closed subdomain of $S_{1}\cup S_{2}$, and $\psi_{\ell}(y)$ are two
linearly independent solutions to the $\epsilon=0$ Orr-Sommerfeld equation
(\ref{Orr-S}). In particular, here $\phi_{j}$ ($j=1,2,3,4$) are four linearly
independent analytic solutions on the entire real interval $[0,1]$. The
eigenvalues of the Orr-Sommerfeld operator (\ref{rwos}) are given by the zeros
of the determinant
\[
\Delta(c,\gamma)=\left\vert
\begin{array}
[c]{lccr}%
\phi_{1}(0) & \phi_{2}(0) & \phi_{3}(0) & \phi_{4}(0)\cr\phi_{1}^{\prime}(0) &
\phi_{2}^{\prime}(0) & \phi_{3}^{\prime}(0) & \phi_{4}^{\prime}(0)\cr\phi
_{1}(1) & \phi_{2}(1) & \phi_{3}(1) & \phi_{4}(1)\cr\phi_{1}^{\prime}(1) &
\phi_{2}^{\prime}(1) & \phi_{3}^{\prime}(1) & \phi_{4}^{\prime}(1)\cr
\end{array}
\right\vert .
\]
Each entry of $\Delta(c,\gamma)$ is analytic in $c\in B_{\delta_{1}}(c^{0})$,
($c^{0}\neq1/2$). Wasow obtained the following expression \cite{Was48},
\[
\Delta(c,\gamma)=\gamma^{2}e^{\gamma(Q(0)+Q(1))}\left[  K\Delta_{0}%
(c)+\O (\gamma^{-1})\right]  ,
\]
where
\[
\Delta_{0}(c)=\left\vert
\begin{array}
[c]{lr}%
\psi_{3}(0) & \psi_{4}(0)\cr\psi_{3}(1) & \psi_{4}(1)\cr
\end{array}
\right\vert ,
\]
and $Re(\gamma Q(0))>0$, $Re(\gamma Q(1))>0$, $K$ is a non-zero constant. Let
\[
\Delta_{1}(c,\gamma)=K\Delta_{0}(c)+\O (\gamma^{-1}).
\]
Notice that the unstable eigenvalue $c^{0}$ of the linear 2D Euler operator is
a zero of $\Delta_{0}(c)$. Here we fix $\alpha=\alpha^{0}$. By a proper choice
of $\delta_{1}$, $\Delta_{0}(c)$ is non-zero on the boundary of $B_{\delta
_{1}}(c^{0})$. Then when $\gamma$ is sufficiently large, by Rouch\'{e}
theorem, $\Delta_{1}(c,\gamma)$ also has a zero $c^{\ast}$ near $c^{0}$. As
$\gamma\rightarrow\infty$, $c^{\ast}\rightarrow c^{0}$. The proof is complete.
\end{proof}

\begin{remark}
The key point in the above proof is to show how to embed the entire real
interval [$0,1$] inside the interior of the two sectors $S_{1}\cup S_{2}$.
This key point was missing in \cite{Mor52}.
\end{remark}

Next we will develop an expression for the eigenfunction needed for later
studies. First we need a lemma.

\begin{lemma}
$Re(\gamma Q(y))$ is a strictly monotone function on $y\in\lbrack0,1]$.
\end{lemma}

\begin{proof}%
\begin{align*}
\frac{d}{dy}Re(\gamma Q(y))  &  =\frac{d}{dy}[Re(\gamma Q(y))-Re(\gamma
Q(0))]\\
&  =\sqrt{\alpha R}\frac{d}{dy}Re\int_{0}^{y}\sqrt{c_{i}+i(U_{n}-c_{r})}ds\\
&  =\sqrt{\alpha R}Re\sqrt{c_{i}+i(U_{n}-c_{r})}.
\end{align*}
For the above to be zero, we need $c_{i}\leq0$ and $U_{n}-c_{r}=0$. Since
$c_{i}>0$ in our case, the above is never zero This proves the lemma.
\end{proof}

The eigenfunction of the Orr-Sommerfeld operator corresponding to the unstable
eigenvalue given by the above theorem is given by
\[
\phi=\left\vert
\begin{array}
[c]{lccr}%
\phi_{1}(y) & \phi_{2}(y) & \phi_{3}(y) & \phi_{4}(y)\cr\phi_{1}^{\prime}(0) &
\phi_{2}^{\prime}(0) & \phi_{3}^{\prime}(0) & \phi_{4}^{\prime}(0)\cr\phi
_{1}(1) & \phi_{2}(1) & \phi_{3}(1) & \phi_{4}(1)\cr\phi_{1}^{\prime}(1) &
\phi_{2}^{\prime}(1) & \phi_{3}^{\prime}(1) & \phi_{4}^{\prime}(1)\cr
\end{array}
\right\vert ,
\]
where $\phi_{j}$ ($j=1,2,3,4$) are given by (\ref{asol1})-(\ref{asol2}). We
need to know the asymptotic property of $\phi$ as $\gamma\rightarrow\infty$.
For this, the only trouble maker is the exponent $\gamma Q(y)$. We choose the
convention
\begin{equation}
Re(\gamma Q(0))<0<Re(\gamma Q(1)). \label{cvt}%
\end{equation}
Some of the entries in the expression of $\phi$ are exponentially small in
$\gamma$. Dropping these entries, we have
\begin{align}
\phi &  =\left\vert
\begin{array}
[c]{lccr}%
\phi_{1}(y) & \phi_{2}(y) & \phi_{3}(y) & \phi_{4}(y)\cr0 & \phi_{2}^{\prime
}(0) & \phi_{3}^{\prime}(0) & \phi_{4}^{\prime}(0)\cr\phi_{1}(1) & 0 &
\phi_{3}(1) & \phi_{4}(1)\cr\phi_{1}^{\prime}(1) & 0 & \phi_{3}^{\prime}(1) &
\phi_{4}^{\prime}(1)\cr
\end{array}
\right\vert \nonumber\\
&  =\left\vert
\begin{array}
[c]{lr}%
\phi_{3}(1) & \phi_{4}(1)\cr\phi_{3}^{\prime}(1) & \phi_{4}^{\prime}(1)\cr
\end{array}
\right\vert \phi_{2}^{\prime}(0)\phi_{1}(y)-\left\vert
\begin{array}
[c]{lcr}%
0 & \phi_{3}^{\prime}(0) & \phi_{4}^{\prime}(0)\cr\phi_{1}(1) & \phi_{3}(1) &
\phi_{4}(1)\cr\phi_{1}^{\prime}(1) & \phi_{3}^{\prime}(1) & \phi_{4}^{\prime
}(1)\cr
\end{array}
\right\vert \phi_{2}(y)\nonumber\\
&  -\left\vert
\begin{array}
[c]{lr}%
\phi_{1}(1) & \phi_{4}(1)\cr\phi_{1}^{\prime}(1) & \phi_{4}^{\prime}(1)\cr
\end{array}
\right\vert \phi_{2}^{\prime}(0)\phi_{3}(y)+\left\vert
\begin{array}
[c]{lr}%
\phi_{1}(1) & \phi_{3}(1)\cr\phi_{1}^{\prime}(1) & \phi_{3}^{\prime}(1)\cr
\end{array}
\right\vert \phi_{2}^{\prime}(0)\phi_{4}(y) \label{EFE}%
\end{align}
The dominant terms are the third and the fourth terms, which are of order
\[
\O \left(  \gamma^{2}e^{Re(\gamma Q(1))-Re(\gamma Q(0))}\right)
\]
under the convention (\ref{cvt}). By rescaling $\phi$,
\[
\phi_{\ast}=\gamma^{-2}e^{Re(\gamma Q(0))-Re(\gamma Q(1))}\phi;
\]
we see that
\[
\phi_{\ast}=C[\phi_{4}(1)\phi_{3}(y)-\phi_{3}(1)\phi_{4}(y)]+\O (\gamma
^{-1}).
\]
Notice that the quantity $\phi_{4}(1)\phi_{3}(y)-\phi_{3}(1)\phi_{4}(y)$ is
not zero, but small of order $\O (\gamma^{-1})$. It depends upon $c^{\ast}$.
When $\gamma\rightarrow\infty$, it does approach zero. Taking derivatives of
$\phi$ in $y$, the balance of orders shifts to the first two terms in
(\ref{EFE}). One can clearly see that
\[
\Vert\phi_{\ast}\Vert_{H^{s}}\sim\O \left(  |\gamma|^{s-1}\right)
,\quad\text{as }\gamma\rightarrow\infty.
\]
In fact, it is a matter of scaling. If we scale $\phi$ as
\[
\phi_{\ast}=\gamma^{-N}e^{Re(\gamma Q(0))-Re(\gamma Q(1))}\phi
\]
for any $N$, then
\[
\Vert\phi_{\ast}\Vert_{H^{s}}\sim\O \left(  |\gamma|^{s-N+1}\right)
,\quad\text{as }\gamma\rightarrow\infty.
\]

\section{A Bifurcation to Kelvin's Cat's Eyes}

The bifurcation we are seeking will stem out of the oscillatory shear
(\ref{Os}). Since there is no external parameter in 2D Euler equation, the
bifurcation here is with regard to an internal parameter: the wave number
along $x$-direction. The proof of Theorem \ref{OST} shows that there is a
neutral mode $\left(  \alpha,c,\phi\right)  =\left(  \alpha_{n},\frac{1}%
{2},\phi_{n}\right)  $ to the Rayleigh equation (\ref{Rayleigh}) associated
with the oscillatory shear $U_{n}\left(  y\right)  \ $(\ref{formula-U_n}).
Here, $-\alpha_{n}^{2}=\lambda_{1}$ is the negative eigenvalue of the operator
$L=-\frac{d^{2}}{dy^{2}}+Q(y)$ with $Q(y)$ defined by (\ref{ExQ}), and
$\phi_{n}$ is the corresponding eigenfunction. By the proof of Theorem
\ref{OST}, $\alpha_{n}\geq c_{0}n$.\ By Remark \ref{remark-inviscid-insta} 3),
we can choose $\phi_{n}\left(  y\right)  $ such that $\phi_{n}\left(
y\right)  >0$ when $y\in\left(  0,1\right)  $. The following theorem shows
that we can get nontrivial inviscid travelling waves bifurcating from the
above neutral modes.

\begin{theorem}
\label{thm-catseye}There is a local bifurcation curve of travelling wave
solutions to the 2D Euler equation with the stream function $\psi=\psi
^{\alpha}(\alpha\left(  x-\frac{1}{2}t\right)  ,y)$, which stems out of the
oscillatory shear $U_{n}\left(  y\right)  \ $(\ref{Os}). Here $\alpha$ is near
$\alpha_{n}\geq c_{0}n$ (Theorem \ref{OST}),
\[
\psi^{\alpha}(\xi,y)\in C^{2,\beta}\left(  \left(  0,2\pi\right)
\times\left(  0,1\right)  \right)  \ (0<\beta<1)
\]
is periodic and even in $\xi$ of period $2\pi$, and $\psi^{\alpha}$ is
constant on $\left\{  y=0\right\}  $ and $\left\{  y=1\right\}  $. Near
$y=1/2$, the streamlines of these new fixed points have a Kelvin cat's eye
structure, with a leading order expression given by (\ref{cats-eye}).
\end{theorem}

\begin{proof}
Let $\xi=\alpha x$ where $\alpha$ will serve as the bifurcation parameter. Let
$\psi_{\text{rel}}(\xi,y)$ to be the relative stream function in the reference
frame $\left(  x-\frac{1}{2}t,y\right)  $, that is,
\[
\psi_{\text{rel}}(\xi,y)=\psi(\xi,y)-\frac{1}{2}y.
\]
The travelling waves solves the Poisson's equation
\begin{equation}
F(\psi,\alpha^{2})\equiv\alpha^{2}\frac{\partial^{2}\psi_{\text{rel}}%
}{\partial\xi^{2}}+\frac{\partial^{2}\psi_{\text{rel}}}{\partial y^{2}}%
-f(\psi_{\text{rel}})=0 \label{Bmap}%
\end{equation}
for some function $f$, with the boundary conditions that $\psi_{\text{rel}}$
takes constant values on $\left\{  y=0\right\}  $ and $\left\{  y=1\right\}
$. Since we seek nontrivial travelling waves near the oscillatory shear
(\ref{Os}), we demand that the oscillatory shear (\ref{Os}) satisfies
(\ref{Bmap}) too. That is
\begin{equation}
\frac{\partial^{2}\psi_{\text{rel}}^{\ast}}{\partial y^{2}}=f(\psi
_{\text{rel}}^{\ast}) \label{eqn-shear}%
\end{equation}
where
\[
\psi_{\text{rel}}^{\ast}\left(  y\right)  =\frac{1}{2}\left(  y-\frac{1}%
{2}\right)  ^{2}-\frac{A}{4n^{2}\pi}\cos\left(  4n\pi y\right)  +\frac
{A}{4n^{2}\pi}%
\]
is the relative stream function for the the oscillatory shear (\ref{Os}). It
follows from (\ref{eqn-shear}) that
\begin{equation}
f^{\prime}(\psi_{\text{rel}}^{\ast})=\frac{U_{n}^{\prime\prime}(y)}%
{U_{n}(y)-\frac{1}{2}}=Q(y), \label{eqn-shear-1}%
\end{equation}
where $Q(y)$ is defined before in (\ref{ExQ}). Since both $\psi_{\text{rel}%
}^{\ast}\left(  y\right)  $ and $Q(y)$ are symmetric to the line $y=\frac
{1}{2}$, we only need to check (\ref{eqn-shear-1}) for $0\leq y\leq\frac{1}%
{2}$. The function $\psi_{\text{rel}}^{\ast}\left(  y\right)  $ is monotone on
$\left[  0,\frac{1}{2}\right]  $ since $\psi_{\text{rel}}^{\ast\prime}\left(
y\right)  =U_{n}(y)-\frac{1}{2}<0$ on $[0,\frac{1}{2})$. Therefore we have
\[
f^{\prime}=Q\circ\left(  \psi_{\text{rel}}^{\ast}\right)  ^{-1},
\]
which determines the relation $f$ in (\ref{Bmap}). We define
\[
\phi\left(  \xi,y\right)  =\psi_{\text{rel}}(\xi,y)-\frac{1}{2}\left(
y-\frac{1}{2}\right)  ^{2}%
\]
and reduce (\ref{Bmap}) to solve the equation
\begin{equation}
\alpha^{2}\frac{\partial^{2}\phi}{\partial\xi^{2}}+\frac{\partial^{2}\phi
}{\partial y^{2}}+1-f(\phi+\frac{1}{2}\left(  y-\frac{1}{2}\right)  ^{2})=0
\label{eqn-phi-traveling}%
\end{equation}
with the homogeneous boundary conditions
\[
\phi(\xi,0)=\phi(\xi,1)=0.
\]
For $\beta\in\left(  0,1\right)  ,\ $define the spaces%
\begin{align*}
B  &  = \bigg \{ \phi(\xi,y)\in C^{2,\beta}([0,2\pi]\times\lbrack0,1]),\\
&  \phi(\xi,0)=\phi(\xi,1)=0,\ 2\pi-\text{periodic and even in }\xi\bigg \}
\end{align*}
and
\[
D=\left\{  \phi(\xi,y)\in C^{0,\beta}([0,2\pi]\times\lbrack0,1]),\text{ }%
2\pi-\text{periodic and even in }\xi\right\}  .
\]
Consider the mapping
\[
F(\phi,\alpha^{2})\ :\ B\times\mathbb{R}^{+}\mapsto D
\]
defined by
\[
F(\phi,\alpha^{2})=\alpha^{2}\frac{\partial^{2}\phi}{\partial\xi^{2}}%
+\frac{\partial^{2}\phi}{\partial y}+1-f(\phi+\frac{1}{2}\left(  y-\frac{1}%
{2}\right)  ^{2}).
\]
Then the travelling wave solutions satisfy the equation $F(\phi,\alpha^{2}%
)=0$. The trivial solutions corresponding to the oscillatory shears (\ref{Os})
have
\[
\phi_{\ast}\left(  y\right)  =-\frac{A}{4n^{2}\pi}\cos\left(  4n\pi y\right)
+\frac{A}{4n^{2}\pi}.
\]
Let $-\alpha_{n}^{2}\ $be the negative eigenvalue of $-\frac{\partial^{2}%
}{\partial y}+Q(y)$ and $\phi_{n}(y)$ the corresponding positive
eigenfunction. The linearized operator of $F$ around$\ \left(  \phi_{\ast
},\alpha_{n}^{2}\right)  $ has the form
\begin{align*}
{\mathcal{L}}  &  :=F_{\psi}(\phi_{\ast},\alpha_{n}^{2})=\alpha_{n}^{2}%
\frac{\partial^{2}}{\partial\xi^{2}}+\frac{\partial^{2}}{\partial y}%
-f^{\prime}(\psi_{\text{rel}}^{\ast})\\
&  =\alpha_{n}^{2}\frac{\partial^{2}}{\partial\xi^{2}}+\frac{\partial^{2}%
}{\partial y}-Q(y).
\end{align*}
Then by Remark \ref{remark-inviscid-insta} 3), the kernel of ${\mathcal{L}}:$
$\ B\mapsto D\ $is given by
\[
\ker({\mathcal{L}})=\left\{  \phi_{n}(y)\cos\xi\right\}  ,
\]
In particular, the dimension of $\ker$(${\mathcal{L}}$) is 1. Since
${\mathcal{L}}$ is self-adjoint, $\phi_{n}(y)\cos\xi\not \in R({\mathcal{L}})$
-- the range of ${\mathcal{L}}$. In fact, again by Remark
\ref{remark-inviscid-insta} 3),
\[
\text{dim}\{B/R({\mathcal{L}})\}=1.
\]
Notice that $\partial_{\alpha^{2}}\partial_{\phi}F(\phi,\alpha^{2})$ is
continuous and
\[
\partial_{\alpha^{2}}\partial_{\psi}F(\phi_{\ast},\alpha_{n}^{2})\left(
\phi_{n}(y)\cos\xi\right)  =\frac{\partial^{2}}{\partial\xi^{2}}\left[
\phi_{n}(y)\cos\xi\right]  =-\phi_{n}(y)\cos\xi\not \in R({\mathcal{L}}).
\]
Therefore by the Crandall-Rabinowitz local bifurcation theorem \cite{CR71},
there is a local bifurcating curve ($\phi(\beta),\alpha^{2}(\beta)$) of
$F(\phi,\alpha^{2})=0$, which intersects the trivial curve ($\phi_{\ast}$,
$\alpha^{2}$) at $\alpha^{2}=\alpha_{n}^{2}$, such that
\[
\phi(\beta)=\phi_{\ast}(y)+\beta\phi_{n}(y)\cos\xi+o(\beta),
\]
and $\alpha^{2}(\beta)$ is a continuous function, $\alpha^{2}(0)=\alpha
_{n}^{2}$. So the relative stream function has the form
\begin{equation}
\psi_{\text{rel}}^{\alpha\left(  \beta\right)  }(\xi,y)=\frac{1}{2}\left(
y-\frac{1}{2}\right)  ^{2}+\phi_{\ast}(y)+\beta\phi_{n}(y)\cos\xi+o(\beta).
\label{cats-eye}%
\end{equation}
Since $\phi_{n}(y)>0$ in $\left(  0,1\right)  $, so near the inflection point
$y=1/2$ of the oscillatory shear (\ref{formula-U_n}), the streamlines of these
travelling waves have a cat's eye structure (see \cite{drazin-reid}) with
saddle points near $\left(  \frac{1}{2},2\pi j\right)  $ $\left(
j\in\mathbf{Z}\right)  $. The proof is complete.
\end{proof}

\begin{remark}
1) The small travelling waves constructed in Theorem \ref{thm-catseye} has a
$x-$period of the order $O\left(  \frac{1}{n}\right)  $. In particular, the
cat's eyes near the oscillatory shear $U_{n}\left(  y\right)  \ $has the
spatial scale $\frac{1}{n}$.

2) Since the oscillatory shears $U_{n}\left(  y\right)  $ (\ref{Os}) are
arbitrarily close to the linear shear in $L^{2}$ norm of velocity, the
travelling waves constructed in Theorem \ref{thm-catseye} appear in an
arbitrarily small ($L^{2}-$velocity) neighborhood of the linear shear. We note
that there might not exist nontrivial travelling waves near the linear shear
in a stronger norm in lieu of Romanov's nonlinear stability theorem
\cite{Rom73}. This conjecture is partly supported by the following rough
argument. Any travelling wave of 2D Euler equation satisfies the Poisson's
equation
\begin{equation}
-\Delta\psi=g\left(  \psi\right)  \label{eqn-Poisson}%
\end{equation}
for some function $g$ in $\Omega=\left(  0,L\right)  \times\left(  0,1\right)
$, where $L$ is the $x$-period and $\psi$ is the relative stream function.
Taking $x$ derivative of (\ref{eqn-Poisson}), we get
\[
-\Delta\psi_{x}=g^{\prime}\left(  \psi\right)  \psi_{x}.
\]
Note that $\psi_{x}=0$ on the boundaries $\left\{  y=0\right\}  $ and
$\left\{  y=1\right\}  $. Multiplying above by $\psi_{x}$ and integration by
parts in $\Omega$, we get
\[
\int\int_{\Omega}\left\vert \nabla\psi_{x}\right\vert ^{2}dxdy=\int
\int_{\Omega}g^{\prime}\left(  \psi\right)  \left\vert \psi_{x}\right\vert
^{2}dxdy.
\]
If the travelling wave is close to the linear shear in a strong norm (e.g.
$C^{1}$-vorticity), $g$ should be close to $1$ in $C^{1}$ norm and thus
$\left\vert g^{\prime}\left(  \psi\right)  \right\vert _{L^{\infty}}$ is very
small. We have
\[
\int\int_{\Omega}\left\vert \nabla\psi_{x}\right\vert ^{2}dxdy\leq\left\vert
g^{\prime}\left(  \psi\right)  \right\vert _{L^{\infty}}\int\int_{\Omega
}\left\vert \psi_{x}\right\vert ^{2}dxdy.
\]
But we also have
\[
\int\int_{\Omega}\left\vert \nabla\psi_{x}\right\vert ^{2}dxdy=\left(
-\Delta\psi_{x},\psi_{x}\right)  \geq\pi^{2}\int\int_{\Omega}\left\vert
\psi_{x}\right\vert ^{2}dxdy.\text{\ }%
\]
Since $\psi_{x}=0$ on $\partial\Omega$ and the operator $-\Delta$ with
Dirichlet boundary conditions on $\partial\Omega$ has the lowest eigenvalue
$\pi^{2}$. So if $\left\vert g^{\prime}\left(  \psi\right)  \right\vert
_{L^{\infty}}<\pi^{2}$, we must have $\psi_{x}\equiv0$ in $\Omega\ $and the
travelling wave is a trivial shear flow.

%Zlin-0510

\end{remark}

\begin{remark}
\label{rmk-TW-critical}After this paper, the study of Euler traveling waves
near Couette was extended in \cite{lin-zeng} to show that (vorticity)
$H^{\frac{3}{2}}$ is the critical regularity for the existence of nontrivial
traveling waves near Couette. More precisely, it is shown in \cite{lin-zeng}
that there exist cats's eyes flows in any (vorticity) $H^{s}$ $\left(
s<\frac{3}{2}\right)  \ $neighborhood of Couette with arbitrary minimal
$x-$period, and no nontrivial traveling waves exist in a sufficiently small
(vorticity) $H^{s}$ $\left(  s>\frac{3}{2}\right)  \ $neighborhood of Couette.
These results shed some light on another puzzle about Couette flow, namely the
nonlinear inviscid damping, for which the linear damping was first discovered
by Orr (\cite{orr}) in 1907.
\end{remark}

\section{Inviscid Linear Instability of 3D Shears}

In this section, we show that the instability of oscillatory shears
$U_{n}\left(  y\right)  $ (\ref{formula-U_n}) persists under the 3D setting.
Consider a 3D shear flow $\vec{u}_{0}=\left(  U\left(  y,z\right)
,0,0\right)  $, which is a steady solution of 3D Euler equation for any
profile $U\left(  y,z\right)  $. The fluid domain is
\[
\Omega_{3}=\left\{  \left(  x,y,z\right)  \ |\ 0<y<1\text{, }L_{x}\text{ and
}L_{z}\ \text{periodic in }x\text{ and }z\right\}  ,\text{ }%
\]
where $L_{x}$ and $L_{z}$ are to be determined later. The linearized 3D Euler
equations near $\vec{u}_{0}$ are
\begin{equation}
\partial_{t}u+Uu_{x}+vU_{y}+wU_{z}=-P_{x}, \label{eqn-u-3d-shear}%
\end{equation}%
\begin{equation}
\partial_{t}v+Uv_{x}=-P_{y},\ \partial_{t}w+Uw_{x}=-P_{z},
\label{eqn-v,w-3d-shear}%
\end{equation}%
\begin{equation}
u_{x}+v_{y}+w_{z}=0, \label{eqn-div-3d-shear}%
\end{equation}
with the boundary conditions
\begin{equation}
v\left(  x,0,z\right)  =v\left(  x,1,z\right)  =0\text{.}
\label{eqn-bc-3d-shear}%
\end{equation}
Here $\left(  u,v,w\right)  $ and $P$ are perturbations of the velocity and
pressure. Consider a normal mode solution $e^{i\alpha\left(  x-ct\right)
}\left(  u,v,w\right)  \left(  y,z\right)  $ to the linearized equation, with
$\alpha=k\frac{2\pi}{L_{x}}$ $\left(  k=1,2,\cdots\right)  $. By eliminating
$u$ and $P$, we end up with a PDE system
\begin{equation}
\left(  U-c\right)  \left(  v_{yy}-\alpha^{2}v+w_{yz}\right)  -U_{yy}%
v-U_{yz}w-U_{z}w_{y}+U_{y}w_{z}=0, \label{Ray-3d-1}%
\end{equation}%
\begin{equation}
\left(  U-c\right)  \left(  w_{zz}-\alpha^{2}w+v_{yz}\right)  -U_{zz}%
w-U_{yz}v-U_{y}v_{z}+U_{z}v_{y}=0, \label{Ray-3d-2}%
\end{equation}
with the boundary conditions $v\left(  0,z\right)  =v\left(  1,z\right)  =0$.
When $U$ depends only on $y$ and $w=0$, the system (\ref{Ray-3d-1}%
)-(\ref{Ray-3d-2}) is reduced to the Rayleigh equation (\ref{Rayleigh}) for 2D
shears. So far, the understanding of the instability of such 3D shears remains
very limited due to the complicated nature of (\ref{Ray-3d-1})-(\ref{Ray-3d-2}%
). Our next theorem shows instability of 3D shears close to an unstable 2D shear.

\begin{theorem}
\label{tm-robust}Let $U_{0}\left(  y\right)  \in C^{2}\left(  0,1\right)  $ be
such that the Rayleigh equation (\ref{Rayleigh}) has an unstable solution with
$\left(  \alpha_{0},c_{0}\right)  $ $\left(  \alpha_{0},\operatorname{Im}%
c_{0}>0\right)  $. Fixed $L_{z}>0$, consider $U\left(  y,z\right)  \in
C^{1}\left(  \left(  0,1\right)  \times\left(  0,L_{z}\right)  \right)  $,
$L_{z}$-periodic in $z$ and
\begin{equation}
U\left(  1,z\right)  =U_{0}\left(  1\right)  ,U\left(  0,z\right)
=U_{0}\left(  0\right)  . \label{bc-3d-shear}%
\end{equation}
If $\left\Vert U\left(  y,z\right)  -U_{0}\left(  y\right)  \right\Vert
_{W^{1,p}\left(  \left(  0,1\right)  \times\left(  0,L_{z}\right)  \right)  }$
$\left(  p>2\right)  \ $is small enough, then there exists an unstable
solution $e^{i\alpha_{0}\left(  x-ct\right)  }\left(  u,v,w,P\right)  \left(
y,z\right)  $ to the linearized equation around $\left(  U\left(  y,z\right)
,0,0\right)  $ with $\left\vert c-c_{0}\right\vert $ small. Moreover, if
$U\left(  y,z\right)  \in C^{\infty}$, then $\left(  u,v,w,P\right)  \in
C^{\infty}$.
\end{theorem}

The proof of Theorem \ref{tm-robust} is divided into several steps. First, we
give a new formulation of linearized growing modes for 3D shears. Fixed
$\alpha=\alpha_{0}.$ Consider a growing mode solution $e^{i\alpha_{0}\left(
x-ct\right)  }\left(  u,v,w,P\right)  \left(  y,z\right)  $ $\left(
\operatorname{Im}c>0\right)  $ to the linearized Euler equation around
$\left(  U\left(  y,z\right)  ,0,0\right)  $. Instead of studying
(\ref{Ray-3d-1})-(\ref{Ray-3d-2}), we reformulate the problem in the following
way. From (\ref{eqn-u-3d-shear})-(\ref{eqn-bc-3d-shear}), we have%
\begin{equation}
i\alpha_{0}\left(  U-c\right)  u+vU_{y}+wU_{z}=-i\alpha_{0}P,
\label{eqn-growing-u}%
\end{equation}%
\begin{equation}
i\alpha_{0}\left(  U-c\right)  v=-P_{y},\ i\alpha_{0}\left(  U-c\right)
w=-P_{z}, \label{eqn-growing-v,w}%
\end{equation}%
\begin{equation}
i\alpha_{0}u+v_{y}+w_{z}=0, \label{eqn-growing-div}%
\end{equation}
with the boundary conditions
\begin{equation}
v\left(  0,z\right)  =v\left(  1,z\right)  =0\text{.} \label{eqn-growing-bc}%
\end{equation}
Denote $\Omega_{2}=\left\{  \left(  y,z\right)  \ |\ 0<y<1\text{, }%
L_{z}\ \text{periodic in }z\right\}  $,
\[
\partial\Omega_{2}=\left\{  y=0\right\}  \cup\left\{  y=1\right\}  ,
\]%
\[
\nabla_{2}=\left(  \partial_{y},\partial_{z}\right)  ,\ \Delta_{2}%
=\partial_{yy}+\partial_{zz},\
\]
and $\vec{u}_{2}=\left(  v,w\right)  $. We claim that:%
\begin{equation}
\int\int_{\Omega_{2}}u\ dydz=0, \label{formula-0-integral-u}%
\end{equation}
and
\begin{equation}
\oint_{\left\{  y=0\right\}  }\vec{u}_{2}\cdot dl=\oint_{\left\{  y=1\right\}
}\vec{u}_{2}\cdot dl=0. \label{formula-0-circulation}%
\end{equation}
The identity (\ref{formula-0-integral-u}) follows by integrating
(\ref{eqn-growing-div}) in $\Omega_{2}$ with the boundary condition
(\ref{bc-3d-shear}). The identity (\ref{formula-0-circulation}) is a result of
(\ref{eqn-growing-v,w}), (\ref{bc-3d-shear}) and the assumption that
$\operatorname{Im}c>0$. Define
\begin{equation}
\omega=w_{y}-v_{z}. \label{defn-vorticity}%
\end{equation}
Then by equation (\ref{eqn-growing-v,w}),
\begin{equation}
\left(  U-c\right)  \omega+U_{y}w-U_{z}v=0. \label{eqn-growing-vorticity}%
\end{equation}
Taking $\left(  y,z\right)  $ divergence of (\ref{eqn-growing-v,w}) and using
(\ref{eqn-growing-div}), we get
\begin{equation}
-\Delta_{2}P=\alpha_{0}^{2}\left(  U-c\right)  u+i\alpha_{0}\left(
U_{y}v+U_{z}w\right)  . \label{eqn-P-poisson}%
\end{equation}
From (\ref{eqn-growing-bc}) and (\ref{eqn-growing-v,w}), $P$ satisfies the
Neumann boundary conditions
\[
P_{y}\left(  0,z\right)  =P_{y}\left(  1,z\right)  =0
\]
on $\partial\Omega_{2}$. Denote
\[
L^{2;0}\left(  \Omega_{2}\right)  =\left\{  f\in L^{2}\left(  \Omega
_{2}\right)  \ |\ \int\int_{\Omega_{2}}f\ dydz=0\right\}  ,
\]%
\[
H^{2;0}\left(  \Omega_{2}\right)  =\left\{  f\in H^{2}\left(  \Omega
_{2}\right)  \ |\ \int\int_{\Omega_{2}}f\ dydz=0\right\}
\]
and $\mathcal{Q}:L^{2}\rightarrow L^{2;0}$ to be the projector to the mean
zero space $L^{2;0}$. For any $f\in L^{2;0}$, denote $h=\left(  -\Delta
_{2}\right)  _{N}^{-1}f$ \ to be the unique solution in $H^{2;0}\left(
\Omega_{2}\right)  $ of the Neumann problem%
\[
-\Delta_{2}h=f,\ \ \ \text{in }\Omega_{2}%
\]%
\[
\frac{\partial h}{\partial n}=0\text{ on }\partial\Omega_{2}.
\]
Thus from (\ref{eqn-P-poisson}), we get
\begin{equation}
\mathcal{Q}P=\left(  -\Delta_{2}\right)  _{N}^{-1}\left[  \alpha_{0}%
^{2}\left(  U-c\right)  u+i\alpha_{0}\left(  U_{y}v+U_{z}w\right)  \right]  .
\label{formula1-P}%
\end{equation}
Denote $\mathcal{B}=\left(  -\Delta_{2}\right)  _{N}^{-1}\mathcal{Q=Q}\left(
-\Delta_{2}\right)  _{N}^{-1}\mathcal{Q}$, then $\mathcal{B}:L^{2}\left(
\Omega_{2}\right)  \rightarrow H^{2;0}\left(  \Omega_{2}\right)  $ is a
self-adjoint, bounded and nonnegative operator. We rewrite (\ref{formula1-P})
as
\begin{equation}
\mathcal{Q}P=\alpha_{0}^{2}\mathcal{B}\left(  Uu\right)  -c\alpha_{0}%
^{2}\mathcal{B}u+i\alpha_{0}\mathcal{B}\left(  U_{y}v+U_{z}w\right)  .
\label{formula2-P}%
\end{equation}
Multiplying (\ref{eqn-growing-u}) by $\mathcal{Q}$ and using the equation
(\ref{formula2-P}), we get
\begin{align}
&  -i\alpha_{0}c\left(  1+\alpha_{0}^{2}\mathcal{B}\right)  u\label{eqn-u}\\
&  =-i\alpha_{0}\left(  \alpha_{0}^{2}\mathcal{B}+\mathcal{Q}\right)  \left(
Uu\right)  +\left(  \alpha_{0}^{2}\mathcal{B}-\mathcal{Q}\right)  \left(
U_{y}v+U_{z}w\ \right)  ,\ \ \ \ \ \ \nonumber
\end{align}
where we use the property $\mathcal{Q}u=u$ due to (\ref{formula-0-integral-u}%
). Let $\lambda=-i\alpha_{0}c,$ then $\operatorname{Re}\lambda>0$. From
(\ref{eqn-u}) and (\ref{eqn-growing-vorticity}), we get
\begin{align}
&  \lambda\left(
\begin{array}
[c]{c}%
u\\
\omega
\end{array}
\right)  =-i\alpha_{0}U\left(
\begin{array}
[c]{c}%
u\\
\omega
\end{array}
\right) \label{reformulation-growing}\\
&  +\left(
\begin{array}
[c]{c}%
\left(  1+\alpha_{0}^{2}\mathcal{B}\right)  ^{-1}\left[  i\alpha_{0}\left(
1-\mathcal{Q}\right)  \left(  Uu\right)  +\left(  \alpha_{0}^{2}%
\mathcal{B}-\mathcal{Q}\right)  \left(  U_{y}v+U_{z}w\ \right)  \right] \\
-i\alpha_{0}\left(  U_{y}w-U_{z}v\right)
\end{array}
\right) \nonumber\\
&  =-i\alpha_{0}U\left(
\begin{array}
[c]{c}%
u\\
\omega
\end{array}
\right)  +\nonumber\\
&  \left(
\begin{array}
[c]{c}%
\frac{i\alpha_{0}}{L_{z}}\int\int_{\Omega_{2}}Uu\ dydz+\left(  1+\alpha
_{0}^{2}\mathcal{B}\right)  ^{-1}\left(  \alpha_{0}^{2}\mathcal{B}%
-\mathcal{Q}\right)  \left(  U_{y}v+U_{z}w\ \right) \\
-i\alpha_{0}\left(  U_{y}w-U_{z}v\right)
\end{array}
\right)  ,\nonumber
\end{align}
since
\begin{align*}
&  \left(  1+\alpha_{0}^{2}\mathcal{B}\right)  ^{-1}\left(  1-\mathcal{Q}%
\right)  \left(  Uu\right)  =\\
&  \left(  \frac{1}{L_{z}}\int\int_{\Omega_{2}}Uu\ dydz\right)  \ \left(
1+\alpha_{0}^{2}\mathcal{B}\right)  ^{-1}1=\frac{1}{L_{z}}\int\int_{\Omega
_{2}}Uu\ dydz.
\end{align*}
Thus the growing mode problem is reduced to study the unstable spectrum of the
operator $\mathcal{A}=\mathcal{F}+\mathcal{K},\ $where $\mathcal{F}$ is the
$-i\alpha_{0}U\ $multiplying operator and
\[
\mathcal{K}\left(
\begin{array}
[c]{c}%
u\\
\omega
\end{array}
\right)  :=\left(
\begin{array}
[c]{c}%
\frac{i\alpha_{0}}{L_{z}}\int\int_{\Omega_{2}}Uu\ dydz+\left(  1+\alpha
_{0}^{2}\mathcal{B}\right)  ^{-1}\left(  \alpha_{0}^{2}\mathcal{B}%
-\mathcal{Q}\right)  \left(  U_{y}v+U_{z}w\ \right) \\
-i\alpha_{0}\left(  U_{y}w-U_{z}v\right)
\end{array}
\right)  .
\]
In the above definition, $\left(  w,v\right)  \in\left(  H^{1}\left(
\Omega_{2}\right)  \right)  ^{2}$ is uniquely determined from $\left(
u,\omega\right)  \in\left(  L^{2}\left(  \Omega_{2}\right)  \right)  ^{2}$ by
solving equations%
\begin{equation}
v_{y}+w_{z}=-i\alpha_{0}\mathcal{Q}u, \label{eqn-Q-div}%
\end{equation}%
\[
w_{y}-v_{z}=\omega,
\]
with the zero circulation condition (\ref{formula-0-circulation}) and the zero
normal velocity condition $v=0$ on $\partial\Omega_{2}.$This is guaranteed by
Lemma \ref{lemma-div-curl} in Appendix.

We study properties of $\mathcal{A}$ in the next lemma.

\begin{lemma}
\label{lemma-property-A}(i) $\mathcal{A}:\left(  L^{2}\left(  \Omega
_{2}\right)  \right)  ^{2}\rightarrow\left(  L^{2}\left(  \Omega_{2}\right)
\right)  ^{2}$ is compact perturbation of $\mathcal{F}.$

(ii) The essential spectrum of $\mathcal{A}$ is $i\left[  \alpha_{0}\min
U,\alpha_{0}\min U\right]  .$
\end{lemma}

\begin{proof}
(ii) is a corollary of (i) because of Weyl's theorem (see \cite{hislop-siegel}
or \cite{kato-book}) and the fact that the operator $\mathcal{F}\ $is bounded,
skew-adjoint and has the essential spectrum $i\left[  \alpha_{0}\min
U,\alpha_{0}\min U\right]  $. To show (i), we need to prove that
$\mathcal{K}:\left(  L^{2}\left(  \Omega_{2}\right)  \right)  ^{2}%
\rightarrow\left(  L^{2}\left(  \Omega_{2}\right)  \right)  ^{2}$ is compact.
By Lemma \ref{lemma-div-curl}, we have
\begin{align*}
\left\Vert \left(
\begin{array}
[c]{c}%
v\\
w
\end{array}
\right)  \right\Vert _{H^{1}\left(  \Omega_{2}\right)  }  &  =\left\Vert
\vec{u}_{2}\ \right\Vert _{H^{1}\left(  \Omega_{2}\right)  }\leq C_{0}\left(
\left\Vert \operatorname{div}\vec{u}_{2}\ \right\Vert _{L^{2}}+\left\Vert
\operatorname{curl}\vec{u}_{2}\ \right\Vert _{L^{2}}\right) \\
&  \leq C_{0}\left(  \alpha_{0}\left\Vert u\ \right\Vert _{L^{2}}+\left\Vert
\ \omega\right\Vert _{L^{2}}\right)  .
\end{align*}
Thus the linear mapping $\left(  u,\omega\right)  \rightarrow\left(
v,w\right)  $ is compact in $\left(  L^{2}\left(  \Omega_{2}\right)  \right)
^{2}$. Since
\[
\left(  1+\alpha_{0}^{2}\mathcal{B}\right)  ^{-1},\mathcal{B},\ \mathcal{Q}%
\]
are bounded, this proves the compactness of $\mathcal{K}.$
\end{proof}

As a corollary of the above lemma, any eigenvalue $\lambda$ of $\mathcal{A}%
\ $with $\operatorname{Re}\lambda>0$ is a discrete eigenvalue with finite multiplicity.

By using the formulation (\ref{reformulation-growing}), the proof of
persistence of instability for 3D shears is similar to that in
\cite{friedlander-et-robust} for 2D Euler equation. Given $\lambda$ with
$\operatorname{Re}\lambda>0$ and $U\left(  y,z\right)  \in C^{1}\left(
\left(  10,1\right)  \times\left(  0,L_{z}\right)  \right)  $, we define the
operator
\[
\mathcal{M}\left(  \lambda,U\right)  :=\left(  \lambda-\mathcal{F}\right)
^{-1}\mathcal{K}%
\]
in $\left(  L^{2}\left(  \Omega_{2}\right)  \right)  ^{2}$. Then the growing
mode equation is reduced to solve $\mathcal{M}\left(
\begin{array}
[c]{c}%
u\\
\omega
\end{array}
\right)  =\left(
\begin{array}
[c]{c}%
u\\
\omega
\end{array}
\right)  $ for some $\lambda$ with $\operatorname{Re}\lambda>0$.

\begin{lemma}
\label{lemma-property-M}Consider $\lambda\in\mathbf{C}^{+}=\left\{
\operatorname{Re}\lambda>0\right\}  $ and $U\left(  y,z\right)  $ satisfying
conditions in Theorem \ref{tm-robust}. Then (i) $\mathcal{M}\left(
\lambda,U\right)  :\left(  L^{2}\left(  \Omega_{2}\right)  \right)
^{2}\rightarrow\left(  L^{2}\left(  \Omega_{2}\right)  \right)  ^{2}$ is
compact and analytical in $\lambda\in\mathbf{C}^{+}$. (ii) $\mathcal{M}\left(
\lambda,U\right)  \ $depends continuously on $U$ in the following sense. Let
$V\left(  y,z\right)  $ satisfy the same conditions of $U\left(  y,z\right)  $
as in Theorem \ref{tm-robust} . Then for any $b>0$, there exists another
constant $C^{\prime}>0$ such that
\begin{equation}
\sup_{\operatorname{Re}\lambda\geq b}\left\Vert \mathcal{M}\left(
\lambda,U\right)  -\mathcal{M}\left(  \lambda,V\right)  \right\Vert
_{\mathcal{L}\left(  \left(  L^{2}\left(  \Omega_{2}\right)  \right)
^{2}\right)  }\leq C^{\prime}\left\Vert U-V\right\Vert _{W^{1,p}}.
\label{norm-continuity}%
\end{equation}

\end{lemma}

\begin{proof}
Since $\mathcal{F}$ is skew-adjoint, for any $\lambda$ with $\operatorname{Re}%
\lambda>0$, $\left(  \lambda-\mathcal{F}\right)  ^{-1}$ is bounded and
$\mathcal{M}$ is compact in $\left(  L^{2}\left(  \Omega_{2}\right)  \right)
^{2}$ by Lemma \ref{lemma-property-A}. Since $\mathcal{F}$ generates an
unitary group and
\[
-\left(  \lambda-\mathcal{F}\right)  ^{-1}=\int_{0}^{+\infty}e^{-\left(
\lambda-\mathcal{F}\right)  t}dt,
\]
so $\mathcal{M}\left(  \lambda,U\right)  $ is analytic in $\lambda$ in the
half-plane $\mathbf{C}^{+}.$ To show (ii), we write
\begin{align*}
&  \mathcal{M}\left(  \lambda,U\right)  -\mathcal{M}\left(  \lambda,V\right)
\\
&  =\left(  \lambda-\mathcal{F}\left(  \lambda,U\right)  \right)
^{-1}\mathcal{K}\left(  \lambda,U\right)  -\left(  \lambda-\mathcal{F}\left(
\lambda,V\right)  \right)  ^{-1}\mathcal{K}\left(  \lambda,V\right) \\
&  =\left(  \lambda-\mathcal{F}\left(  \lambda,U\right)  \right)  ^{-1}\left(
\mathcal{F}\left(  \lambda,U\right)  -\mathcal{F}\left(  \lambda,V\right)
\right)  \left(  \lambda-\mathcal{F}\left(  \lambda,V\right)  \right)
^{-1}\mathcal{K}\left(  \lambda,U\right) \\
&  +\left(  \lambda-\mathcal{F}\left(  \lambda,V\right)  \right)  ^{-1}\left(
\mathcal{K}\left(  \lambda,U\right)  -\mathcal{K}\left(  \lambda,V\right)
\right) \\
&  =I+II.
\end{align*}
When $\operatorname{Re}\lambda\geq b,$ $\left\Vert \left(  \lambda
-\mathcal{F}\left(  \lambda,U\right)  \right)  ^{-1}\right\Vert \leq\frac
{1}{b}$. Both $\mathcal{K}\left(  \lambda,U\right)  $ and $\mathcal{F}\left(
\lambda,U\right)  $ are norm continuous to $U$ in the norm $\left\Vert
U\right\Vert _{W^{1,p}}$, which we show below. We use $C$ for a generic
constant$.$ First,
\[
\left\Vert \mathcal{F}\left(  \lambda,U\right)  -\mathcal{F}\left(
\lambda,V\right)  \right\Vert \leq\alpha_{0}\left\Vert U-V\right\Vert
_{L^{\infty}}\leq C\left\Vert U-V\right\Vert _{W^{1,p}}.
\]
Second, for any $\left(  u,\omega\right)  \in\left(  L^{2}\left(  \Omega
_{2}\right)  \right)  ^{2},$we have
\begin{align*}
&  \left\Vert \left(  \mathcal{K}\left(  \lambda,U\right)  -\mathcal{K}\left(
\lambda,V\right)  \right)  \left(
\begin{array}
[c]{c}%
u\\
\omega
\end{array}
\right)  \right\Vert _{L^{2}}\\
&  \leq C\left(  \left\Vert U-V\right\Vert _{L^{\infty}}\left\Vert
u\right\Vert _{L^{2}}+\left\Vert \nabla\left(  U-V\right)  \right\Vert
_{L^{p}}\left\Vert \left(
\begin{array}
[c]{c}%
v\\
w
\end{array}
\right)  \right\Vert _{L^{\frac{2p}{p-2}}}\right) \\
&  \leq C\left(  \left\Vert U-V\right\Vert _{L^{\infty}}\left\Vert
u\right\Vert _{L^{2}}+\left\Vert \nabla\left(  U-V\right)  \right\Vert
_{L^{p}}\left\Vert \left(
\begin{array}
[c]{c}%
v\\
w
\end{array}
\right)  \right\Vert _{H^{1}}\right) \\
&  \leq C\left\Vert U-V\right\Vert _{W^{1,p}}\left(  \left\Vert \left(
\begin{array}
[c]{c}%
u\\
\omega
\end{array}
\right)  \right\Vert _{L^{2}}\right)
\end{align*}
and thus
\[
\left\Vert \left(  \mathcal{K}\left(  \lambda,U\right)  -\mathcal{K}\left(
\lambda,V\right)  \right)  \right\Vert _{L^{2}}\leq C\left\Vert U-V\right\Vert
_{W^{1,p}}.
\]
So both $I$ and $II$ above are uniformly controlled by $\left\Vert
U-V\right\Vert _{W^{1,p}}$. This proves (\ref{norm-continuity}).
\end{proof}

The proof of Theorem \ref{tm-robust} uses the following Lemma of Steinberg
\cite{steinberg}.

\begin{lemma}
\label{lemma-steinberg}Let $T(\lambda;s)$ be a family of compact operators on
a Banach space analytic in $\lambda$ and jointly continuous in $(\lambda;s)$
for each $(\lambda;s)\in\Lambda\times$ $S$, where $\Lambda\ $is an open set in
$\mathbb{C}$ and $S$ is an interval in $\mathbf{R}$. If for each $s$ there
exists a $\lambda\ $such that $I-T(\lambda;s)$ is invertible, then $\left(
I-T(\lambda;s)\right)  ^{-1}$ is meromorphic in $\lambda$ for each $s$ and the
poles of $\left(  I-T(\lambda;s)\right)  ^{-1}$ depend continuously on $s$ and
can appear or disappear only at the boundary of $\Lambda\ $or at infinity.
\end{lemma}

\begin{proof}
[Proof of Theorem \ref{tm-robust}]By definition, $\lambda$ is an unstable
eigenvalue of $\mathcal{A}$ if and only if $1$ is an eigenvalue of
$\mathcal{M}\left(  \lambda,U\right)  $. For $\operatorname{Re}\lambda>0$ and
$0\leq s\leq1$, define
\[
\mathcal{T}\left(  \lambda,s\right)  =\left(  1-s\right)  \mathcal{M}\left(
\lambda,U_{0}\right)  +s\mathcal{M}\left(  \lambda,U\right)  .
\]
By Lemma \ref{lemma-property-M}, these operators are compact, analytic in
$\lambda$ and continuous in $s$. The assumption on $U_{0}\left(  y\right)  $
implies that $\lambda_{0}=-i\alpha_{0}c_{0}$ is a pole of $\left(
I-\mathcal{T}\left(  \lambda,0\right)  \right)  ^{-1}$ in the right half
plane. Since the poles of $\left(  I-\mathcal{T}\left(  \lambda,0\right)
\right)  ^{-1}$ are discrete, we can choose $\varepsilon_{0}$ so small such
that on the circle $\Gamma=\left\{  \lambda\ |\ \left\vert \lambda-\lambda
_{0}\right\vert \leq\varepsilon_{0}\right\}  $, the inverse $\left(
I-\mathcal{T}\left(  \lambda,0\right)  \right)  ^{-1}$ exists. By the
continuity property (\ref{norm-continuity}),
\[
\left\Vert \mathcal{T}\left(  \lambda,s\right)  -\mathcal{T}\left(
\lambda,0\right)  \right\Vert \leq s\left\Vert \mathcal{M}\left(
\lambda,U_{0}\right)  -\mathcal{M}\left(  \lambda,U\right)  \right\Vert \leq
C\left\Vert U-U_{0}\right\Vert _{W^{1,p}}.
\]
Thus when $\left\Vert U-U_{0}\right\Vert _{W^{1,p}}$ is sufficiently small,
$\left(  I-\mathcal{T}\left(  \lambda,s\right)  \right)  ^{-1}$ also exists on
the circle $\Gamma$ for all $s\in\left[  0,1\right]  $. Thus by Lemma
\ref{lemma-steinberg}, there exists a pole $\lambda_{1}$ of $\left(
I-\mathcal{T}\left(  \lambda,1\right)  \right)  ^{-1}=\left(  I-\mathcal{M}%
\left(  \lambda,U\right)  \right)  ^{-1}$ within the disk $\left\{  \left\vert
\lambda-\lambda_{0}\right\vert <\varepsilon_{0}\right\}  $. For $\varepsilon
_{0}$ small, $\lambda_{1}$ also has positive real part and is an unstable
eigenvalue of the operator $\mathcal{A}$ associated with $U\left(  y,z\right)
$. Let $\left(  u,\omega\right)  \in\left(  L^{2}\left(  \Omega_{2}\right)
\right)  ^{2}\ $be the corresponding eigenfunction and $\left(  v,w\right)
\in\left(  H^{1}\left(  \Omega_{2}\right)  \right)  ^{2}$ is determined by
$\left(  u,\omega\right)  $ as in the definition of $\mathcal{K}$. Define
\[
c=\frac{\lambda_{1}}{-i\alpha_{0}},\ P=\frac{1}{-i\alpha_{0}}\left(
i\alpha_{0}\left(  U-c\right)  u+vU_{y}+wU_{z}\right)  .
\]
We check that $\left(  u,v,w,P\right)  $ solves the growing mode equations
(\ref{eqn-growing-u})-(\ref{eqn-growing-bc}). First, an integration of the
equation $\mathcal{A}u=\lambda_{1}u$ implies that $\int\int_{\Omega_{2}%
}u\ dydz=0$, namely $\mathcal{Q}u=u$. So (\ref{eqn-growing-div}) follows from
(\ref{eqn-Q-div}). The equations (\ref{eqn-growing-u}) and
(\ref{eqn-growing-bc}) are already implied in our construction. To check
(\ref{eqn-growing-v,w}), first we note that the vector field $i\alpha
_{0}\left(  U-c\right)  \left(  v,w\right)  $ is curl free because of the
equation (\ref{eqn-growing-vorticity}) which follows from $\mathcal{A}%
\omega=\lambda_{1}\omega$. So there exists a scalar function $P^{\prime
}\left(  y,z\right)  $ such that
\[
i\alpha_{0}\left(  U-c\right)  \left(  v,w\right)  =-\left(  P_{x}^{\prime
},P_{y}^{\prime}\right)  .
\]
Taking divergence of above and using (\ref{eqn-growing-div}), we can set
\[
P^{\prime}=\left(  -\Delta_{2}\right)  _{N}^{-1}\left(  \alpha_{0}^{2}\left(
U-c\right)  u+i\alpha_{0}\left(  U_{y}v+U_{z}w\right)  \right)
\]
by modulating a constant. It remains to show (\ref{formula1-P}), from which
(\ref{eqn-growing-v,w}) follows. We note that $\mathcal{A}u=\lambda_{1}u$
implies (\ref{eqn-u}). Combining with $\mathcal{Q}u=u$, we get
(\ref{formula2-P}), an equivalent form of (\ref{formula1-P}).

To show the regularity of the growing mode, we look at the eigenfunction
equation
\[
\mathcal{M}\left(
\begin{array}
[c]{c}%
u\\
\omega
\end{array}
\right)  =\left(  \lambda_{1}-\mathcal{F}\right)  ^{-1}\mathcal{K}\left(
\begin{array}
[c]{c}%
u\\
\omega
\end{array}
\right)  =\left(
\begin{array}
[c]{c}%
u\\
\omega
\end{array}
\right)  .
\]
The eigenfunction $\left(  u,\omega\right)  \in\left(  L^{2}\left(  \Omega
_{2}\right)  \right)  ^{2}$ implies that $\left(  v,w\right)  \in\left(
H^{1}\left(  \Omega_{2}\right)  \right)  ^{2}$ and thus $\mathcal{K}\left(
\begin{array}
[c]{c}%
u\\
\omega
\end{array}
\right)  \in\left(  H^{1}\left(  \Omega_{2}\right)  \right)  ^{2}.$ Since
$\left(  \lambda_{1}-\mathcal{F}\right)  ^{-1}$ is regularity preserving, we
have $\left(  u,\omega\right)  \in\left(  H^{1}\left(  \Omega_{2}\right)
\right)  ^{2}$. If $U\in C^{\infty}$, we can repeat this process to deduce
that $\left(  u,\omega\right)  \in\left(  H^{k}\left(  \Omega_{2}\right)
\right)  ^{2}$ for any $k\geq1$ and therefore $\left(  u,\omega\right)
\in\left(  C^{\infty}\left(  \Omega_{2}\right)  \right)  ^{2}$. Then
$u,v,w,P\in C^{\infty}\left(  \Omega_{2}\right)  $. This finishes the proof of
Theorem \ref{tm-robust}.
\end{proof}

\begin{remark}
In Theorem \ref{tm-robust}, the smallness of $\left\Vert U\left(  y,z\right)
-U_{0}\left(  y\right)  \right\Vert _{W^{1,p}\left(  \Omega_{2}\right)  }$
$\left(  p>2\right)  $ is required to show persistence of instability. For 2D
shears, i.e. $U=U\left(  y\right)  $, we can show that: if $\left\Vert
U\left(  y\right)  -U_{0}\left(  y\right)  \right\Vert _{H^{1}\left(
0,1\right)  }$ is small enough, then linear instability of $U_{0}\left(
y\right)  $ implies that of $U\left(  y\right)  $. This is an improvement over
the result in \cite{friedlander-et-robust}, where the smallness in $C^{2}$
norm is required to prove persistence of instability for 2D Euler flows
without hyperbolic points. The proof is very similar to the 3D case, although
much simpler. So we only sketch it briefly. In the 2D case, we reformulate the
growing mode problem in terms of $u$. Using the notations as before, from the
linearized equations we derive
\[
\lambda u=-i\alpha_{0}Uu+i\alpha_{0}\int_{0}^{1}Uu\ dy+\left(  1+\alpha
_{0}^{2}\mathcal{B}\right)  ^{-1}\left(  \alpha_{0}^{2}\mathcal{B}%
-\mathcal{Q}\right)  \left(  U_{y}v\right)  .
\]
Here, $\mathcal{Q\,}$\ is the projector from $L^{2}\left(  0,1\right)  $ to
mean zero space and $\mathcal{B=Q}\left(  -\frac{d^{2}}{dy^{2}}\right)
_{N}^{-1}\mathcal{Q}$ where $\left(  -\frac{d^{2}}{dy^{2}}\right)  _{N}^{-1}$
is the mean zero solution operator of the 1D Neumann problem$.$The mapping
$u\rightarrow v$ is defined by solving the ODE
\[
\frac{dv}{dy}=-i\alpha_{0}\mathcal{Q}u,\ v\left(  0\right)  =v\left(
1\right)  =0.
\]
The rest of the proof is the same as in the 3D case, so we skip it. Smallness
of $\left\Vert U\left(  y\right)  -U_{0}\left(  y\right)  \right\Vert
_{H^{1}\left(  0,1\right)  }$ is required in the proof because of the Sobolev
embedding $H^{1}\left(  0,1\right)  \hookrightarrow L^{\infty}\left(
0,1\right)  .$
\end{remark}

\section{Appendix}

\begin{proof}
[Proof of the claim $\lambda_{2}\geq0$ in Remark \ref{remark-inviscid-insta}
3):]By the Sturm-Liouville theorem \cite{Zet05}, the linear operator $L$,
\[
L\varphi=-\varphi^{\prime\prime}+Q(y)\varphi,\quad\varphi(0)=\varphi(1)=0
\]
has a sequence of real eigenvalues
\[
\lambda_{1}<\lambda_{2}<\lambda_{3}<\cdots\rightarrow+\infty,
\]
and there is a unique eigenfunction $\varphi_{n}$ corresponding to each
eigenvalue $\lambda_{n}$, which has exactly $n-1$ zeros in $(0,1)$. Moreover,
$\{\varphi_{n}\}$ form an orthonormal base of $L_{[0,1]}^{2}$ under the
boundary condition $\varphi(0)=\varphi(1)=0$. Thus
\[
\lambda_{1}=\min_{\varphi}\frac{\langle L\varphi,\varphi\rangle}%
{\langle\varphi,\varphi\rangle},
\]
where $\langle\ ,\ \rangle$ denotes $L^{2}$ inner product. In the case of
Theorem \ref{OST}, we also know that $\lambda_{2}\geq0$ by the following
argument. Suppose otherwise, $\lambda_{2}<0$ and let $\varphi_{2}$ be its
corresponding eigenfunction. Then $\varphi_{2}$ has one zero $z\in(0,1)$. Let
$\phi=U(y)-U(1/2)$, then
\begin{equation}
-\phi^{\prime\prime}+Q\phi=0. \label{phe1}%
\end{equation}
Also
\begin{equation}
-\varphi_{2}^{\prime\prime}+Q\varphi_{2}=\lambda_{2}\varphi_{2}. \label{phe2}%
\end{equation}
Multiplying (\ref{phe2}) by $\phi$ and (\ref{phe1}) by $\varphi_{2}$, then
subtracting, we have
\begin{equation}
-\varphi_{2}^{\prime\prime}\phi+\phi^{\prime\prime}\varphi_{2}=\lambda
_{2}\varphi_{2}\phi. \label{phe3}%
\end{equation}
If $z\leq1/2$, we use the interval $[0,z]$; otherwise we use the interval
$[z,1]$. Without loss of generality, we assume $z\leq1/2$. Integrating
(\ref{phe3}) over the interval $[0,z]$, we have
\begin{equation}
-\left(  \varphi_{2}^{\prime}\phi\right)  |_{0}^{z}=\lambda_{2}\int_{0}%
^{z}\varphi_{2}\phi dy. \label{phe4}%
\end{equation}
On the interval $[0,z]$, neither $\phi$ nor $\varphi_{2}$ changes sign. We
note that $\varphi_{2}^{\prime}(0)$ has the same sign with $\varphi_{2}$ on
[$0,z$], while $\varphi_{2}^{\prime}(z)$ has the opposite sign. Then the right
hand side and the left hand side of (\ref{phe4}) have opposite signs. A
contradiction is reached (even in the case $\varphi_{2}^{\prime}(0)=0$ and/or
$\varphi_{2}^{\prime}(z)=0$). Thus $\lambda_{2}\geq0$.
\end{proof}

The following lemma is used in the proof of Theorem \ref{tm-robust}.

\begin{lemma}
\label{lemma-div-curl}Given $\left(  f_{1},f_{2}\right)  \in\left(
L^{2}\left(  \Omega_{2}\right)  \right)  ^{2}$ and $\int\int_{\Omega_{2}}%
f_{1}\ dydz=0$. Then there exists a unique vector field $\vec{u}=\left(
v,w\right)  \in$ $\left(  H^{1}\left(  \Omega_{2}\right)  \right)  ^{2}$ such
that
\begin{equation}
\operatorname{div}\vec{u}=v_{y}+w_{z}=f_{1}, \label{div-lemma}%
\end{equation}%
\begin{equation}
\operatorname{curl}\vec{u}=w_{y}-v_{z}=f_{2}, \label{curl-lemma}%
\end{equation}%
\begin{equation}
\oint_{\left\{  y=0\right\}  }\vec{u}\cdot dl=\oint_{\left\{  y=1\right\}
}\vec{u}\cdot dl=0, \label{circulation-lemma}%
\end{equation}
and%
\begin{equation}
v\left(  0,z\right)  =v\left(  1,z\right)  =0\text{.}
\label{zero-normal-lemma}%
\end{equation}
Moreover,
\begin{equation}
\left\Vert \vec{u}\right\Vert _{H^{m+1}}\leq C\left(  \left\Vert
f_{1}\right\Vert _{H^{m}}+\left\Vert f_{2}\right\Vert _{H^{m}}\right)
\label{estimate-lemma}%
\end{equation}
for any $m\geq0$.
\end{lemma}

\begin{proof}
First, we show the existence. We split $\vec{u}=\vec{u}_{1}+\vec{u}_{2}$,
where $\vec{u}_{1}$ and $\vec{u}_{2}\ $satisfy
\[
\operatorname{div}\vec{u}_{1}=0,\ \operatorname{curl}\vec{u}_{1}=f_{2},\
\]
and
\[
\operatorname{div}\vec{u}_{2}=f_{1},\ \operatorname{curl}\vec{u}_{2}=0,\
\]
respectively, with conditions (\ref{circulation-lemma}) and
(\ref{zero-normal-lemma}). For existence and uniqueness of $\vec{u}_{1},$ we
refer to Theorem 2.2 of \cite[Chapter 1]{marchiro-pulvirenti}. We construct
$\vec{u}_{2}=\nabla\varphi$, where $\varphi$ solves the Neumann problem
\[
\Delta\varphi=f_{1}\text{ in }\Omega_{2},\ \frac{\partial\varphi}{\partial
n}=0\text{ on }\partial\Omega_{2}.
\]
The solvability is ensured by the zero mean assumption on $f_{1}$ and Fredholm
alternative principle$.$ Note that condition (\ref{circulation-lemma}) is
automatic for the gradient flow $\vec{u}_{2}$.

Second, we show uniqueness. If there exists two vectors fields $\vec{u}$ and
$\vec{u}^{\prime}$ satisfying (\ref{div-lemma})-(\ref{zero-normal-lemma}%
)$.$Then their difference $\vec{u}^{\prime\prime}=\vec{u}-\vec{u}^{\prime}$ is
an irrotational and divergence-free field, tangent to $\partial\Omega_{2}$
with zero circulations on $\partial\Omega_{2}$. By Theorem 2.1 of
\cite[Chapter 1]{marchiro-pulvirenti}, $\vec{u}^{\prime\prime}=0$ and thus
$\vec{u}=\vec{u}^{\prime}$.

It remains to prove the estimate (\ref{estimate-lemma}). For this, we use the
following general estimate (see (1.26) in \cite[Page 318]{teman-NS}): Let
$\Omega$ be an open bounded domain of $\mathbf{R}^{2}$, $m$ is an integer
$\geq1$. Then for any $\vec{u}\in H^{m}\left(  \Omega\right)  $ with $\vec
{u}\cdot n=0$ on $\partial\Omega$, there exists $c\left(  m,\Omega\right)  $
such that
\begin{equation}
\left\Vert \vec{u}\right\Vert _{H^{m}}\leq C\left(  \left\Vert
\operatorname{div}\vec{u}\right\Vert _{H^{m-1}}+\left\Vert \operatorname{curl}%
\vec{u}\right\Vert _{H^{m-1}}+\left\Vert \vec{u}\right\Vert _{L^{2}}\right)  .
\label{general-ineq-lemma}%
\end{equation}
We shall show that: for any $\vec{u}\in\left(  H^{1}\left(  \Omega_{2}\right)
\right)  ^{2}$ satisfying (\ref{circulation-lemma}) and
(\ref{zero-normal-lemma})$,$%
\begin{equation}
\left\Vert \vec{u}\right\Vert _{L^{2}\left(  \Omega_{2}\right)  }\leq
c_{0}\left(  \left\Vert \operatorname{div}\vec{u}\right\Vert _{L^{2}\left(
\Omega_{2}\right)  }+\left\Vert \operatorname{curl}\vec{u}\right\Vert
_{L^{2}\left(  \Omega_{2}\right)  }\right)  . \label{div-curl--lemma}%
\end{equation}
Then (\ref{estimate-lemma}) is obvious from (\ref{general-ineq-lemma}) and
(\ref{div-curl--lemma}). We prove (\ref{div-curl--lemma}) by a contradiction
argument. Suppose otherwise, for any $n\geq1$, there exists $\vec{u}_{n}%
\in\left(  H^{1}\left(  \Omega_{2}\right)  \right)  ^{2}$ such that
\begin{equation}
\left\Vert \vec{u}_{n}\right\Vert _{L^{2}\left(  \Omega_{2}\right)  }\geq
n\left(  \left\Vert \operatorname{div}\vec{u}_{n}\right\Vert _{L^{2}\left(
\Omega_{2}\right)  }+\left\Vert \operatorname{curl}\vec{u}_{n}\right\Vert
_{L^{2}\left(  \Omega_{2}\right)  }\right)  . \label{contra-lemma}%
\end{equation}
we normalize $\left\Vert \vec{u}_{n}\right\Vert _{L^{2}\left(  \Omega
_{2}\right)  }=1$. Then by (\ref{general-ineq-lemma}), $\left\Vert \vec{u}%
_{n}\right\Vert _{H^{1}\left(  \Omega_{2}\right)  }$ is uniformly bounded. So
$\vec{u}_{n}$ converges to $\vec{u}_{\infty}$ weakly in $H^{1}\left(
\Omega_{2}\right)  $ and strongly in $L^{2}\left(  \Omega_{2}\right)  $. Thus
$\left\Vert \vec{u}_{\infty}\right\Vert _{L^{2}\left(  \Omega_{2}\right)  }%
=1$. But from (\ref{contra-lemma}), $\vec{u}_{\infty}$ is irrotational and
divergence-free. Moreover, $\vec{u}_{\infty}$ also satisfies
(\ref{circulation-lemma}) and (\ref{zero-normal-lemma})$.$ So $\vec{u}%
_{\infty}=0$, a contradiction. This finishes the proof.
\end{proof}

\begin{center}
{\Large Acknowledgement}
\end{center}

This work is supported partly by the NSF grant DMS-0908175 (Z. Lin) and the 
DoE grant DE-FG02-06ER46307 (Y. Li).

\end{document}